\documentclass{amsart}

\usepackage[top=1.2in, bottom=1.2in, left=1.0in, right=1.0in]{geometry}

\usepackage{amsmath}
\usepackage{amssymb}
\usepackage{mathtools}
\usepackage{amsthm}
\usepackage{amsmath}
\usepackage{amsfonts}
\usepackage{amssymb}
\usepackage{graphicx}
\usepackage{subfigure}
\usepackage{placeins}
\usepackage{enumerate}
\usepackage[foot]{amsaddr}

\def\diag{\mathrm{diag}}
\def\us{u^{*}}
\def\RR{\mathbb R}
\def\NN{\mathbb N}
\def\eps{\varepsilon}
\def\mm{\mathcal{M}}
\def\nn{\mathcal{N}}

\def\jj{\mathcal{J}}
\def\diam{\mathrm{diam}}

\def\d{\mathrm{d}}

\def\de{\partial}
\def\remspace{\kern-0.5em}
\def\half{\frac{1}{2}}

\newcommand{\R}{\mathbb{R}}
\newcommand{\mF}{\mathbf{F}}
\newcommand{\ihalf}{{i+1/2}}
\newcommand{\imhalf}{{i-1/2}}
\newcommand{\partialt}[1]{\dfrac{\partial#1}{\partial t}}
\newcommand{\partialb}[1]{\dfrac{\partial#1}{\partial x_\beta}}
\newcommand{\ddt}{\frac{\mathrm{d}}{\mathrm{d}t}}

\newcommand{\fpartial}[1]{\dfrac{\partial}{\partial #1}}
\newcommand\ds{\displaystyle}
\newcommand{\nm}[1]{\left\| #1 \right\|}

\newcommand{\quot}[1]{``#1''}

\newcommand{\ms}{\medskip}

\newcommand{\amatrix}{\mathcal{A}}
\newcommand{\drift}{\mathcal{B}}

\newlength\eqnspace
\newlength\alnspace
\theoremstyle{definition}
\newtheorem{defi}{Definition}[section]
\newtheorem{rem}{Remark}
\newtheorem*{rem*}{Remark}

\theoremstyle{plain}
\newtheorem{lem}{Lemma}[section]
\newtheorem{prop}[lem]{Proposition}
\newtheorem{theor}[lem]{Theorem}

\begin{document}

\title{Trend to equilibrium for systems with small cross-diffusion}
\author{Luca Alasio$^1$}
\address[$^1$]{Gran Sasso Science Institute, Viale Francesco Crispi 7, L'Aquila, 67100, Italy}
\author{Helene Ranetbauer$^2$}
\address[$^2$]{University of Vienna, Oskar-Morgenstern-Platz 1, 1090 Vienna, Austria}
\author{Markus Schmidtchen$^3$}
\address[$^3$]{Laboratoire Jacques-Louis Lions, Sorbonne-Universit\'e, 4 place Jussieu, 75252 Paris}
\author{Marie-Therese Wolfram$^4$}
\address[$^4$]{Mathematics Institute, University of Warwick, Gibbet Hill Road, CV47AL Coventry, UK}
\address{\quad\, Radon Institute for Computational and Applied Mathematics, Altenbergerstr. 69, 4040 Linz, Austria.}

\date{}

\subjclass{35B40, 35B45, 35K51, 65N08}
\keywords{cross-diffusion systems, asymptotic behaviour}

\begin{abstract} 
This paper presents new analytical results for a class of nonlinear parabolic systems of partial different equations with small cross-diffusion which describe the macroscopic dynamics of a variety of large systems of interacting particles. Under suitable assumptions, we prove existence of classical solutions and we show exponential convergence in time to the stationary state.
Furthermore, we consider the special case of one mobile and one immobile species, for which the system reduces to a nonlinear equation of Fokker-Planck type. In this framework, we improve the convergence result obtained for the general system and we derive sharper $L^{\infty}$-bounds for the solutions in two spatial dimensions. We conclude by illustrating the behaviour of solutions with numerical experiments in one and two spatial dimensions.
\end{abstract}

\maketitle

\section{Introduction}

\subsection{Background and motivation}

In this paper we focus on a class of nonlinear cross-diffusion systems with sufficiently small off-diagonal diffusion terms. The smallness assumption ensures that the system we consider is \quot{close}, in a suitable sense, to a linear, decoupled system. This allows us to adapt techniques developed in the theory of linear, parabolic systems of partial differential equations (PDEs) in order to study long-time behaviour of solutions of the nonlinear system. 
More specifically, we consider a parabolic system of PDEs of the form

\begin{equation}\label{main0}
    \partialt u 
    -\nabla \cdot \left\{
    D(x) [ (I + \delta \Phi(u)) \nabla u 
    + (\diag(\nabla V) + \delta \Psi(u) ) u ]
    \right\}
    =
    0,
\end{equation}
which is a compact notation for the system
\begin{equation}
    \ds \partialt {u_{i}} - \ds \sum_{\alpha, \beta} \fpartial {x_\alpha} 
    \left\{ D_{i}^{\alpha\beta}(x) 
    \left[ \left(
    \partialb {u_{i}} + \delta \sum_j \Phi_{ij}(u)\partialb {u_{j}}
    \right)
    +
    \left(\partialb {V_{i}} u_i 
    + \delta \sum_j \Psi_{ij}(u)u_{j} \right)
    \right] \right\}  =  0,
\end{equation}
for $1\leq i, j \leq m,$ and $1\leq \alpha, \beta \leq d$. Here $m\geq 1$ represents the number of components (or species) and $d\in\{1,2,3\}$ the number of space dimensions. We denote by $u_i$ the $i$-th species. We will specify the assumptions on the diffusion tensor $D$, nonlinear mobilities  $\Phi$ and $\Psi$ and the potential $V$ in Section 
\ref{sec:model}.

Cross-diffusion systems arise in multiple contexts in Physics, Life Sciences and Social Sciences; in particular, they have been derived as formal macroscopic limits of several microscopic models describing multi-species systems in the presence of finite volume effects, size exclusion or joint population pressures
(see, for example, \cite{Bruna:2012cg, Burger:2010gb, perthame2015parabolic, simpson2009multi}). Models with finite volume effects, which ensure physical bounds on the density, received significant attention in the past years. Recently Gavish et al. derived a mean-field model for a one-dimensional hard rod system rigorously, see \cite{gavish2019large}. However these techniques can  be used in 1D only; in higher space dimensions only approximate models have been developed so far. These mean field models share common features, such as degenerate diffusion and small cross-diffusion terms (see, e.g., Examples 1.1-1.3 in \cite{alasio2018stability}). They can be derived from a lattice based microscopic description, see, for instance, \cite{bodnar2005} and a  subsequent formal Taylor expansion of the associated master equation, as considered in \cite{burger2012nonlinear, burger2016}. Another derivation was performed in \cite{Bruna:2012cg, Bruna:2012wu}, where the authors derived a cross-diffusion system from an underlying stochastic microscopic representation using the method of matched asymptotic expansions. We highlight that the two approaches yield different continuum models, however, they share the property of having small cross-diffusion terms. 
The smallness assumption is justified since the cross-diffusion terms are of the same order of magnitude as the microscopic particle size.

\subsection{Gradient flow techniques}

In recent years, gradient flow methods have been successfully employed to study certain families of cross-diffusion systems, see, among others, \cite{jungel2016entropy, difrancesco2018nonlinear, desvillettes2015entropic}.
Techniques such as the boundedness-by-entropy principle introduced in \cite{jungel2015boundedness}, provide a mathematical framework to ensure existence and uniqueness of solutions to general nonlinear cross-diffusion systems that exhibit a gradient flow structure. In general, entropy methods have been proven to be a very useful tool to analyse the long time behavior of evolution equations. Especially the Bakry-Emery strategy (see \cite{bakry1985}) provides the necessary convex Sobolev inequalities to quantify the trend to equilibrium. Then entropy-entropy dissipation estimates allow to deduce exponential decay rates for general classes of linear and nonlinear scalar evolution equations, see for example \cite{arnold2001convex}.
The cross-diffusion systems mentioned above lack a full gradient flow structure (in the Wasserstein sense) in certain parameter ranges, even though the underlying microscopic system possesses a natural one. This lack is caused by approximations made due to the finite volume effects. A first connection between the large deviations of stochastic particle systems and macroscopic Wasserstein gradient flows was established in \cite{adams2011large} without finite volume constraints. Questions related to structural features of cross-diffusion systems to be interpreted as macroscopic gradient flows (as in \cite{zamponi2017analysis}) or to have a strong solution (see \cite{Berendsen2019}) were investigated rather recently.\\
Unfortunately, the results and techniques mentioned above cannot be applied at this stage to the systems of PDEs we want to study, nevertheless in this work we present an alternative strategy that relies on the smallness of the off-diagonal cross-diffusion terms.

\subsection{Numerical methods}

The development of computational methods for non-linear cross-diffusion systems, especially structure preserving schemes, advanced significantly with the recent analytic progress. Structure preserving methods are designed in such a fashion that they preserve important physical and structural features such as positivity, conservation of mass or the dissipation of the associated entropy. Owing to the fact that Wasserstein gradient flows are posed in the set of probability measures, conservation of mass (or probability) of solutions is an important physical feature and finite volume discretisations are a natural framework to guarantee this. In addition, in recent years, several advances have been made in designing flux approximations that are in agreement with the energy dissipation. Bessemoulin-Chatard and Filbet, see \cite{bessemoulin2012finite}, were among the first to present a finite volume method for nonlinear degenerate parabolic equations, which resolved the long-time behaviour correctly. Based on their scheme, different finite volume schemes have been proposed for systems, see for example \cite{CHS18, carrillo2018fvconvergence}. Other numerical approximations are  based on the underlying variational Wasserstein gradient flow structure. These so-called variational schemes are often restricted to one space dimension, as the computational complexity of computing the Wasserstein distance in higher space dimension is significant, see, for instance, \cite{carrillo2016diffeo}. Also, convergence results are, to the best of our knowledge,  restricted to one spatial dimension; cf. \cite{matthes2014}.

\subsection{Summary of the main results}

We consider a family of systems of PDEs with small cross-diffusion terms (namely system \eqref{main0}) for which existence and uniqueness of solutions were investigated in \cite{alasio2018stability}, as recalled in Proposition \ref{prop:ABC}.
We extend the analysis presented in \cite{alasio2018stability} by showing in Theorem \ref{thm:main1} that (under suitable assumptions) the solutions are global-in-time and classical. Furthermore, we provide insights into the equilibration behaviour of such cross-diffusion systems in Theorem \ref{lem:energysys}. A special instance of system \eqref{main0}, to which we shall return in Section \ref{sec:agf}, was analysed in \cite{bruna2017cross} using gradient flow techniques. This reduced PDE (namely problem \eqref{eq:main-equation}) can be interpreted as  the diffusion of hard-core interacting particles through a domain with obstacles distributed according to the given immobile function. We recall existing results for the full cross-diffusion system and its stationary states in Proposition \ref{prop:stationary} and how they can be improved for the reduced PDE at hand. In addition we present a more general equilibration result as well as sharper $L^{\infty}$-bounds for the reduced non-linear PDE in Theorem \ref{thm1} and Theorem \ref{theor:maximum} respectively. We conclude by illustrating our analytic results with numerical simulations. In doing so we construct a finite volume scheme in one and two spatial dimensions, investigate the exponential rate of convergence to equilibrium numerically and illustrate the dynamics of the system in physically relevant scenarios. 
The impact of the immobile species is clearly visible in that the immobile species diminishes the mobility of the mobile species; cf. Section \ref{sec:numerics}.

The rest of paper is organised as follows: we start by recalling well-posedness results for general systems with small cross-diffusion and prove existence of global classical solutions in Section \ref{sec:model}.  Section \ref{sec:expdec} states the main exponential convergence result for the full system. In Section \ref{sec:agf} we consider a specific example of an asymptotic gradient flow system, describing a mobile species diffusing in the presence of an immobile given species. In particular we are able to prove $L^\infty$-bounds for such scalar, nonlinear equations of Fokker-Planck type.
Numerical simulations illustrating the analytic results are presented in Section \ref{sec:numerics}.

\section{Analysis of systems with small cross-diffusion}\label{sec:model}

We start by introducing the cross-diffusion system under consideration and by discussing analytic properties of solutions. We recall existing well-posedness results and establish existence of classic solutions under certain conditions on the initial datum.

\subsection{A well-posedness result}\label{ssec:wellposed}

Throughout the paper $m\geq 1$ represents the number of components and $d\in\{1,2,3\}$ corresponds to the number of space dimensions.
In addition, we shall denote the $i$-th species by $u_i$ and the related initial distribution by $u_i^0\geq 0$.
We call $\nu$ the outward normal of $\Omega$.

For $1\leq i, j \leq m,$ and $1\leq \alpha, \beta \leq d$, we introduce the following notation:
\begin{align}  
    \label{eq:1011}
    \left\{
    \begin{aligned}
        \amatrix_{ij}^{\alpha\beta}(x,u) 
        &=  D_{i}^{\alpha \beta}(x)
        ( I_{ij}
        +\delta \Phi_{ij}(u)
        ),
        \\[\eqnspace]
        \drift_{ij}^{\alpha}(x,u) 
        &=  \sum_\beta
        D_{i}^{\alpha \beta}(x)
        \left( I_{ij} \partialb {V_j(x)}
        +\delta \Psi_{ij}(u)
        \right),
    \end{aligned}
    \right.
\end{align}
where $\delta$ is a small parameter such that 
$$
0 \leq \delta \leq \delta_0,
$$
for a suitable constant $\delta_0$ depending on $u_0$, $D$, $V$, $\Phi$, $\Psi$, and $\Omega$ (see \cite{alasio2018stability} for further details on the smallness assumption).
We consider the following  initial-boundary value problem for system \eqref{main0}:
\begin{align}
    \label{main-general}
    \left\{\;
    \begin{aligned}
	    \ds \partialt {u_{i}} - \ds \sum_{\alpha, \beta, j} \fpartial {x_\alpha} \left[\amatrix_{ij}^{\alpha\beta}(x,u) \partialb {u_{j}} - \drift_{ij}^{\alpha}(x,u)u_{j}\right] & =  0 &\text{in }& \Omega,\\[\eqnspace]
        \sum_{\alpha, \beta, j} \nu_{\alpha} \left[ \amatrix_{ij}^{\alpha\beta}(x,u)\partialb {u_{j}} - \drift_{ij}^{\alpha}(x,u)u_{j}\right] & =  0 &\text{on }& \partial\Omega,\\[\eqnspace]
        u_i(0,\cdot) & = u_i^{0} &\text{in }& \Omega.
    \end{aligned}
    \right.
\end{align}

Note that, due to the no-flux boundary condition, the total \quot{mass} (or number of particles in the microscopic case) of each species is conserved.

Global existence and regularity of solutions for system \eqref{main-general} was proved in \cite{alasio2018stability}. In this framework, it is essential to assume that the system is close to a diagonal, decoupled, linear problem.  Such \quot{reference problem} is given by the weak formulation of 
\begin{align}
    \label{eq:ModelZero}
    \left\{
    \begin{aligned}
	   \ds  \partialt {u_{i}} - \sum_{\alpha, \beta} \fpartial {x_\alpha}
	    \left[
	    D_{i}^{\alpha\beta}(x)
	        \left(\partialb {u_{i}} - \partialb {V_i} u_{i}\right)
	    \right] 
	    &= \ds 0 &\text{in }&\Omega,\\[\eqnspace]
         \sum_{\alpha, \beta} \nu_{\alpha} \cdot \left[
	        D_{i}^{\alpha\beta}(x)
	        \left(\partialb {u_{i}} - \partialb {V_i} u_{i}\right)
        \right]
        &=  0 &\text{on }&\partial\Omega, \\[\eqnspace]
        u_i(0,\cdot) &=  u_i^{0} &\text{in }& \Omega,
    \end{aligned}
    \right.
\end{align}
for $1\leq i \leq m,$ and $1\leq \alpha, \beta \leq d$.\\

We now state our assumptions in detail (notice that they are strictly related to those in \cite{alasio2018stability}).
\begin{enumerate}
\item[\textbf{H1}]
We assume that  $D_{i}^{\alpha\beta}\in C^1(\bar{\Omega})$ is symmetric in the space indices $\alpha$ and $\beta$.
Furthermore, we suppose that there exist two constants $\Lambda \geq \lambda >0$
such that for every $x\in\Omega$, $1\leq i \leq m$, and $\xi\in\mathbb{R}^{d}$,
it holds that
\begin{equation}
    \label{eq:D-Elliptic}
    \lambda\left|\xi\right|^{2}
    \leq \sum_{\alpha,\beta}
    D_{i}^{\alpha \beta}(x)\xi^{\alpha}\xi^{\beta}
    \leq
    \Lambda\left|\xi\right|^{2},
\end{equation}
for any $i=1,\ldots,m$.
\item[\textbf{H2}]
The dependence on $u$ of the nonlinear terms is (at least) of class $C^2$:
\begin{equation}
    \Phi,\Psi\in C^{2}\left(\RR^{m}\right)^{m\times m},  \qquad \text{and} \qquad \Phi(0)=\Psi(0)=0.\label{eq:LipschitzPerturbationAssumption}
\end{equation}

\item[\textbf{H3}]
We assume that, for $i=1,\dots,m$, the potentials $V_i:\RR^d\to\RR$ are (at least) of class $C^2$.

\item[\textbf{H4}] The domain $\Omega$ is bounded, connected and sufficiently smooth (e.g. of class $C^2$). The initial datum  $u^{0}$  in \eqref{main-general} is assumed to be non-negative and to belong to $H^{2}(\Omega)$  satisfying the following compatibility condition on the boundary, $\partial\Omega$, for $i=1,\ldots,m$:
\begin{equation}
    \label{eq:compatcondu_0}
    \sum_{\alpha,j}
    \left[
    \sum_{\beta}
    \amatrix_{ij}^{\alpha\beta}(x,u^0)\partialb {u^0_{j}}
    - 
    \drift_{ij}^{\alpha}(x,u^0)u^0_{j} \right] \nu_\alpha=0.
\end{equation}
\end{enumerate}

Before recalling the well-posedness result from \cite{alasio2018stability} we introduce the following notation.
\begin{defi}\label{def:weaksol}
We denote by  $W(Q_{T})$ the Banach space of functions with two spatial weak derivatives taking values in $L^{2}(\Omega)$ continuously in time, 
and one time derivative in $L^2(0,T;H^1(\Omega))$, that is, 
\begin{equation*}
    W\left(Q_{T}\right)  = C\left(\left[0,T\right];H^{2}(\Omega)\right)
    \cap H^{1}
    \left( 0,T; H^1(\Omega)\right).
\end{equation*}
Note that, given $T>0$, we denote the parabolic cylinder by $Q_{T}=(0,T)\times\Omega$. 
Also, we write $H^2(\Omega)$ for $H^2(\Omega; \mathbb R^m)$, and similarly for other spaces.
\end{defi}

\begin{prop}[See \cite{alasio2018stability}]
\label{prop:ABC} 
Let hypotheses
\textbf{H1}-\textbf{H4}
hold, then system \eqref{main-general} admits a unique, global solution $u\in W(Q_{T})$ for all $T>0$.
Furthermore, there exist constants $\Gamma_0, \Gamma_1 > 0$ independent of $T$ such that:
\begin{equation}
    \nm{u}_{C^0(Q_{T})}
    \leq
    \Gamma_0 \nm{u}_{W(Q_{T})}
    \leq
    \Gamma_1 \nm{u_0}_{H^2(\Omega)}.
\end{equation}
\end{prop}

\subsection{Existence of global classical solutions}\label{sec:reg}

We are now going to show that solutions (in the sense of Definition \ref{def:weaksol}) of problem \eqref{main-general} are, indeed, classical solutions if the initial data is sufficiently smooth. The arguments in this section follow closely those presented in \cite{alasio2019global} for a different class of parabolic systems.
We start by showing that classical solutions exist for a short time (see Proposition \ref{thm:AT}) and then extend their maximal time interval of definition to the positive half-line using uniform estimates (see Proposition \ref{thm:amann}). This yields the main result of this section:
\begin{theor}
    \label{thm:main1}
    Let hypotheses \textbf{H1}-\textbf{H4} hold and assume that $u_0$ is (at least) of class $C^2(\bar{\Omega})$. Then problem \eqref{main-general} admits a unique, bounded, global, classical solution in the space $ C^0([0,\infty);C^2(\bar{\Omega})) \cap C^1([0,\infty);C^0(\bar{\Omega}))$.
\end{theor}

In order to prove Theorem \ref{thm:main1}, we will use the following two fundamental building blocks:

\begin{prop}[Existence of classical solutions for short time, \cite{acquistapace1988quasilinear}]
\label{thm:AT}
Under the hypotheses of Theorem \ref{thm:main1}, there exists a time $\tau \in(0,T)$ such that, given a sufficiently smooth initial datum $u_0$ that is compatible with the boundary conditions, problem \eqref{main-general} has a unique solution $u$ in the interval $[0,\tau]$ which satisfies
$$
    u\in C^{1+\alpha_0}([0,\tau];L^2(\Omega))\cap C^{\alpha_0}([0,\tau];H^2(\Omega)),
$$
where $\alpha_0\in (0,\frac{1}{2})$ depends on $u_0$, $\amatrix$ and $\drift$.
Furthermore, for all $\alpha_1\in(0,\alpha_0)$, we have
$$
    \partialt u, \; \nabla^2 u \; \in C^{\alpha_1}([0,\tau];C^0(\bar{\Omega})).
$$
\end{prop}

\begin{rem}
We denote by $J(u^0) \subseteq [0,\infty)$ the maximal interval of definition of classical solutions of \eqref{main-general} given an initial datum $u^0$.
\end{rem}

\begin{prop}[Criterion for global existence,  \cite{amann1989dynamic}]\label{thm:amann} Let all assumptions of  Proposition \ref{thm:main1} hold and  consider a solution $u$ of problem \eqref{main-general}. If an exponent $\eps\in (0,1)$ (not depending on time) exists, such  that for any $T>0$,
$$
    u\in C^\eps(J(u_0)\cap[0,T];C^0(\bar{\Omega})),
$$
then $u$ is a global solution, i.e., $J(u_0)=[0,\infty)$.
\end{prop}

\begin{rem}[Notation]
Proposition \ref{thm:amann} is a rephrasing of Theorem 2 in \cite{amann1989dynamic} with $\theta=0$.
Notice also that the \quot{convective term}, indicated by $f$ in the notation of \cite{amann1989dynamic}, is affine in the gradient in our case. Additionally, we do not use the notation $BUC^\eps$ to denote the space of \quot{bounded, uniformly $\eps$-H\"older continuous functions} but simply write that the H\"older exponent does not depend on time.
\end{rem}

The following interpolation results provides sufficient conditions for the criterion in Proposition \ref{thm:amann} to be satisfied.

\begin{lem}\label{lem:fractional}
Let $f: Q_T \to\RR$ be a function in $X = L^2(0,T;H^2(\Omega)) \cap H^1(0,T;L^2(\Omega))$, then
$$
    f\in H^r(0,T;H^s(\Omega)), \;\;\; \text{ for all } \; r,s\geq 0 
    \; \text{ such that } \; r + \frac{s}{2} \leq 1,
$$
and, in turn,
$$
f\in  C^{0,\eta}([0,T];C^{0,\theta}(\bar{\Omega})), 
\; \text{ for all } \; 
\eta,\theta\geq 0 
\; \text{ such that } \; 2\eta + \theta \leq \frac{1}{2}.
$$
\end{lem}

\begin{proof}
Thanks to the higher order extensions for Sobolev functions, we can define $f$ on a larger cylindrical domain $R\subseteq \RR^{d+1}$ containing $Q_T$. Introducing a cut-off function, we further extend $f$ to the whole space ensuring sufficiently fast decay at infinity.
Let us call $g$ such an extension. We observe that the norm of $g$ in $X'=L^2(\RR;H^2(\RR^d))\cap H^1(\RR;L^2(\RR^d))$ is controlled by the corresponding norms of $f$ on $Q_T$.
Let $\langle \kappa \rangle = (1+|\kappa|^2)^{1/2}$. Denoting by $(\omega, \kappa)$ the conjugate variables of $(t,x)$ in Fourier space, we have that
$$
    \langle \omega \rangle \hat{g} \in L^2(\RR^{d+1}), 
    \quad
    \text{as well as}
    \quad 
    \langle \kappa \rangle^2 \hat{g} \in L^2(\RR^{d+1}).
$$
This means that 
$
    \left(
    \langle \omega \rangle + \langle \kappa \rangle^2
    \right) \hat{g} \in L^2(\RR^{d+1})
$
and we obtain
$
\langle \omega \rangle^r \langle \kappa \rangle^s |\hat{g}|
\leq
\left(
\langle \omega \rangle + \langle \kappa \rangle^2
\right)^{r+\frac{s}{2}} |\hat{g}|.
$
Thus we obtain the desired fractional Sobolev regularity provided that $r+\frac{s}{2}\leq 1$. We are also using the following inequality, relating the norms of $g$ and $f$,
$$
\nm{g}_{X'} \leq 2\nm{f}_{X}.
$$
Then the H\"older regularity follows from standard embeddings for fractional Sobolev spaces (see, e.g., \cite{dinezza2012hitchhiker}). In particular, for $r,s>\frac{1}{2}$, we take $\eta = r - \frac{1}{2}$ and $\theta = s - \frac{1}{2}$.
\end{proof}

\begin{proof}[Proof of Theorem \ref{thm:main1}]
Thanks to Proposition \ref{thm:AT}, we know that classical solutions exist for short times and that, as explained in \cite{acquistapace1988quasilinear}, they may be extended to a maximal interval of existence denoted by $J(u_0)$ by standard methods.
In order to show that such solutions exist for arbitrarily large time we are going to use the criterion provided by Proposition \ref{thm:amann}.
In particular, we need H\"older continuity of $u$ with respect to time, as well as a uniform $L^\infty$-bound in the space variable.
Thanks to Proposition \ref{prop:ABC}, we know that $u\in W$ with a bound that is uniform in time. This implies that $u\in X$ and hence we apply Lemma \ref{lem:fractional} in order to obtain obtain uniform H\"older estimates. Thus Proposition \ref{thm:amann} allows us to conclude the proof.
\end{proof}

\subsection{Poincar\'e's inequality}

In the next sections we will need the following simple modification of Poincar\'e inequality (see e.g. \cite[Lemma 1.36]{troianiello}). 
We recall that, if $\Omega$ is convex,  we may take $C_P\leq \diam(\Omega) \pi^{-1}$ (see \cite{payne1960optimal}); otherwise a detailed treatment of this type of inequalities for more general $\Omega$ can be found in  \cite[Chapter 4]{ziemer2012weakly}.

\begin{lem}[A Poincar\'e--type inequality]
\label{lem:poincare}
Consider a bounded, connected, Lipschitz domain $\Omega$, a function $f\in H^1(\Omega)$ and a subset $S\subset\Omega$ with positive Lebesgue measure. Suppose that $\int_S f(x) \d x=0$, then there exists a constant $C_P>0$ (depending on $\Omega$ and $S$) such that
\begin{equation}\label{poincare}
    \nm{f}_{L^2(\Omega)} \leq C_P\nm{\nabla f}_{L^2(\Omega)}.
\end{equation}
Furthermore, given a family of functions $(f_\ell)_{\ell\in\RR} \subset H^1(\Omega)$, such that $\sigma_0 = \mathrm{inf}_{\ell\in\RR} \{
{|\Omega|^{-1}}{|\mathrm{ess\, supp}(f_\ell)|}
\} \in (0,1)$, then inequality \eqref{poincare} holds for any function $f_\ell$ and the constant $C_P$ depends only on $\Omega$ and $\sigma_0$.
\end{lem}
\begin{proof}
We argue by contradiction.
If the statement did not hold, then for each $n>0$ we could find a function $f_n\in H^1(\Omega)$ such that $\int_S f_n(x) \d x=0$ and
$$
\nm{f_n}_{L^2(\Omega)} > n\nm{\nabla f_n}_{L^2(\Omega)}.
$$
Without loss of generality, we can assume that $\nm{f_n}_{L^2(\Omega)}=1$ (by homogeneity of the norm).
This implies that $\nm{\nabla f_n}_{L^2(\Omega)}<n^{-1}$. 
Therefore, as $n\to\infty$, we can extract a subsequence converging strongly in $L^2(\Omega)$ to a limit function $f_\infty\in H^1(\Omega)$ which satisfies the following properties 
$$
    \nm{\nabla f_\infty}_{L^2(\Omega)}= 0, \quad 
    \int_S f_\infty(x) \d x=0,
    \quad \text{and} \quad
    \nm{f_\infty}_{L^2(\Omega)}=1.
$$
The first two conditions above imply that $f_\infty = 0$, which contradicts the third.
The proof of the second statement is very similar and we therefore omit it.
\end{proof}

\begin{rem}\label{rem:P}
Under the same assumptions of Lemma \ref{lem:poincare}, we have
\begin{equation}\label{eq:estim1}
    \nm{\nabla f}_{L^2(\Omega)}
    \leq \frac{1}{1-C_P\nm{\nabla V}_\infty}
    \nm{\nabla f + f\nabla V}_{L^2(\Omega)},
\end{equation}
provided that $C_P\nm{\nabla V}_\infty < 1$.
More explicitly, we have
\begin{align*}
    \nm{\nabla f}_{L^2(\Omega)} 
    &\leq \nm{\nabla f + f\nabla V}_{L^2(\Omega)}
    \,+\, \nm{f\nabla V}_{L^2(\Omega)}
    \\[\alnspace]
    &\leq \nm{\nabla f + f\nabla V}_{L^2(\Omega)}
    \,+\, \nm{\nabla V}_{\infty}\nm{f}_{L^2(\Omega)}
    \\[\alnspace]
    &\leq \nm{\nabla f + f\nabla V}_{L^2(\Omega)}
    \,+\, C_P\nm{\nabla V}_{\infty}\nm{\nabla f}_{L^2(\Omega)}.
\end{align*}
Notice that the condition $C_P\nm{\nabla V}_\infty < 1$ gives a restriction on the size of $\Omega$
(since the Poincar\'e constant depends on $\diam(\Omega)$) or, alternatively, on the $L^\infty$-norm of the drift term $\nabla V$. We will use this relation in estimate \eqref{eq:hLtwoestimate}.
\end{rem}

\section{Trend to equilibrium}\label{sec:expdec}

In this section we investigate the equilibration behaviour of solutions to system \eqref{main-general}. We show that they converge exponentially fast to the stationary state in the $L^2$-norm if assumptions \textbf{H1}-\textbf{H5} are satisfied. \\
We choose a more compact notation in the following. Let $u = (u_1, \ldots u_m)$, then system \eqref{eq:ModelZero} can be written as 
\begin{equation}
    \label{eq:mainsys}
    \left\{
    \begin{array}{rl}
    \partialt u 
    -\nabla \cdot \left\{
    D(x) [ (I + \delta \Phi(u)) \nabla u 
    + (\diag(\nabla V) + \delta \Psi(u) ) u ]
    \right\}\remspace
    &=
    0
    \quad \text{ in } \Omega,
    \\[\eqnspace]
    \nu \cdot \left\{  D(x) [ (I + \delta \Phi(u)) \nabla u 
    + (\diag(\nabla V) + \delta \Psi(u) ) u ] \right\} \remspace
    &= 0
    \quad \text{ on } \de\Omega,
    \\[\eqnspace]
    u(0,x)\remspace
    &=
    u^0(x),
    \end{array}
    \right.
\end{equation}
and the corresponding stationary problem:
\begin{equation}
    \label{eq:mainsysstat}
    \left\{
    \begin{array}{rl}
    \nabla \cdot \left\{ 
    D(x) [ (I + \delta \Phi(\us)) \nabla \us 
    + (\diag(\nabla V) + \delta \Psi(\us) ) \us ]
    \right\}\remspace
    &=
    0
    \quad \text{ in } \Omega,
    \\[\eqnspace]
    \nu \cdot \left\{D(x) [ (I + \delta \Phi(\us)) \nabla \us 
    + (\diag(\nabla V) + \delta \Psi(\us) ) \us ] \right\} \remspace
    &= 0
    \quad \text{ on } \de\Omega,
    \end{array}
    \right.
\end{equation}
where we used the compact notation
$$
    D = \left\{ D_{i}^{\alpha \beta} \right\}^{1\leq\alpha,\beta\leq d}_{ 1\leq i\leq m}, \qquad\text{and}\qquad \diag(\nabla V) = \left\{ \partialb{V_i} \right\}^{1\leq\beta\leq d}_{1\leq i\leq m},
$$
as well as
$$
    \Phi = \{ \Phi_{ij} \}_{1\leq i,j\leq m}, \qquad\text{and}\qquad \Psi = \{ \Psi_{ij} \}_{1\leq i,j\leq m}.
$$

Recall that $\Omega$ is a bounded domain of class $C^2$ in $\RR^d$ (with $d=1,2,3$), and $0\leq\delta\leq \delta_0\leq 1$, $u^0_i\in C^2(\Omega)$ and $V_i\in C^{2}(\bar{\Omega})$. 

Before stating and proving the main theorem of this section, we introduce some notation and simple inequalities that we will use in the proof.

\begin{rem}\label{rem:L}
Given a function $F:\RR^m\to\RR^{m\times m}$ of class at least $C^k$ ($k\geq 0$) such that $F(0)=0$, we define the quantity
$$
    L_k(F,R) := \nm{F}_{C^k(\overline{B_{R}(0)})};
$$
in particular, given two functions $w,z:\RR^d\to\RR^m$ such that $w,z\leq M$ for some $M>0$ and for a.e. $x\in\RR^d$, then we have (by Taylor expansion)
$$
    |F(w)|\leq L_0(F,M), \quad \text{ and } \quad 
    |F(w)-F(z)|\leq L_1(F,M) |w-z|.
$$
\end{rem}

\begin{rem}\label{rem:V}
We shall use the following notation:
$$
    \nm{V}_{\infty} = \max_{1\leq i \leq m} \nm{ V_i}_{L^\infty(\Omega)},
    \quad
    \text{ and } 
    \quad
    \nm{\nabla V}_{\infty} = \max_{1\leq i \leq m} \nm{\nabla V_i}_{L^\infty(\Omega)}.
$$
\end{rem}

We now state the main result of this section:

\begin{theor}[Exponential convergence]\label{lem:energysys}
Let assumptions \textbf{H1}-\textbf{H4} hold and 
suppose that $u\in W(Q_T)$ and $\us \in W^{1,\infty}(\Omega)$ are solutions of problem \eqref{eq:mainsys} and \eqref{eq:mainsysstat} respectively.
Let us denote by $M>0$ a constant not depending on $T$ such that
\begin{equation}\label{eq:Mbound2}
    |u(t,x)|\leq M,
\end{equation}
for any $(t,x)\in[0,T]\times\Omega$. Furthermore, suppose that 
\begin{equation}
\label{eq:deltabound2}
    \delta \leq 
    \min \left\{
    \delta_0,\,\,
    \frac{\lambda (1-C_P\nm{\nabla V}_{\infty})}{2\Lambda e^{\nm{V}_{\infty}}\Gamma_M(1+C_P\nm{\nabla V}_{\infty})}
    \right\},
    \quad
    \text{ and }
    \quad
    C_P\nm{\nabla V}_\infty < 1,
\end{equation}
where $\Gamma_M$ is given by \eqref{const:gamma}.
Then the following inequality holds for any $t_1, t_2 \geq 0$:
\begin{equation}\label{eq:eneryest2}
    \nm{u(t_2) - \us}_{L^2(\Omega)}
    \leq 
    \exp\left(3\nm{V}_\infty- \frac{\lambda}{2C_P^2} (t_2 - t_1)\right)
    \nm{u(t_1) - \us}_{L^2(\Omega)}.
\end{equation}
\end{theor}

\begin{proof}
Let $h := u - \us$, then $h$ satisfies the following system (with no-flux boundary conditions) in $\Omega$:

\begin{equation}
    \label{eq:mainsysh}
    \begin{split}
    \partialt h 
    =\nabla \cdot \left\{
    D(x) \left[ (I + \delta \Phi(u)) \nabla h
    +\delta ( \Phi(u)-\Phi(\us) ) \nabla \us
    + \left(\diag(\nabla V) + \delta \Psi(u) \right) h
    + \delta (\Psi(u) - \Psi(\us))\us\right]
    \right\}.
    \end{split}
\end{equation}
Testing system \eqref{eq:mainsysh} against $h \exp(V)$ gives
\begin{align}
    \label{eq:testhV}
    \ddt \int_{\Omega} \frac{1}{2} \sum_{i} e^{V_i} h_i^2 \d x
    \,+\,
    \int_{\Omega} \sum_{ij} e^{V_i}
    \left[ \nabla h_i + h_i\nabla V_i \right]
    \cdot D_{ij} \left[ \nabla h_i + h_i\nabla V_i 
    + \delta \nn_j
    \right] \d x
    = 0,
\end{align}
where 
\begin{align*}
    \nn_j &= \nn_j^{\mathrm{I}} + \nn_j^{\mathrm{II}} + \nn_j^{\mathrm{III}} + \nn_j^{\mathrm{IV}}
    \\[\alnspace]
    &= \sum_k \left\{
    \Phi_{jk}(u)\nabla h_k
    + [ \Phi_{jk}(u)-\Phi_{jk}(\us) ] \nabla \us_k
    + \Psi_{jk}(u) h_k
    + [\Psi_{jk}(u) - \Psi_{jk}(\us)]\us_{k}
    \right\}.
\end{align*}
Next we estimate the order $\delta$ terms in \eqref{eq:testhV} and define
$$
    \int_{\Omega} \sum_{ij} e^{V_i}
    \left[ \nabla h_i + h_i\nabla V_i \right]
    \cdot D_{ij} 
    \left[
    \nn_j^{\mathrm{I}} + \nn_j^{\mathrm{II}} + \nn_j^{\mathrm{III}} + \nn_j^{\mathrm{IV}}
    \right] 
    \d x
    =
    \jj^{\mathrm{I}} + \jj^{\mathrm{II}} + \jj^{\mathrm{III}} + \jj^{\mathrm{IV}}.
$$
In particular, making extensive use of Remarks \ref{rem:L}, \ref{rem:V}, and \ref{rem:P},  we derive the following four estimates:
\begin{align*}
    \jj^{\mathrm{I}} 
    &= \int_{\Omega} 
    \sum_{ijk} e^{V_i}
    \left[ \nabla h_i + h_i\nabla V_i \right]
    \cdot D_{ij}
     \Phi_{jk}(u)\nabla h_k
    \d x
    \\[\alnspace]
    & \leq
    \Lambda \,L_0(\Phi,M) e^{\nm{V}_{\infty}}\, 
    \left[ 
    \nm{\nabla h}_{L^2(\Omega)}^2
    + \nm{\nabla V}_{\infty}\nm{\nabla h}_{L^2(\Omega)}\nm{h}_{L^2(\Omega)}
    \right]
    \\[\alnspace]
    & \leq
    \Lambda \,L_0(\Phi,M) e^{\nm{V}_{\infty}}\,(1+C_P\nm{\nabla V}_{\infty})\nm{\nabla h}_{L^2(\Omega)}^2,
    \\
    & \quad
    \\[\alnspace]
    \jj^{\mathrm{II}}
    &=\int_{\Omega} 
    \sum_{ijk} e^{V_i}
    \left[ \nabla h_i + h_i\nabla V_i \right]
    \cdot D_{ij}
     ( \Phi_{jk}(u)-\Phi_{jk}(\us) ) \nabla \us_k
    \d x
    \\[\alnspace]
    & \leq
    \Lambda \,L_1(\Phi,M) \,e^{\nm{V}_{\infty}}\, \nm{\nabla \us}_{L^\infty(\Omega)}\,
    \left[ 
    \nm{\nabla h}_{L^2(\Omega)}\nm{h}_{L^2(\Omega)}
    +\nm{\nabla V}_{\infty}\nm{h}_{L^2(\Omega)}^2
    \right]
    \\[\alnspace]
    & \leq
    \Lambda \,L_1(\Phi,M)\, e^{\nm{V}_{\infty}}\,
    \nm{\nabla \us}_{L^\infty(\Omega)}\,
    C_P \,(1+C_P\nm{\nabla V}_{\infty})\nm{\nabla h}_{L^2(\Omega)}^2,
\end{align*}
as well as
\begin{align*}
    \jj^{\mathrm{III}}
    &=\int_{\Omega} 
    \sum_{ijk} e^{V_i}
    \left[ \nabla h_i + h_i\nabla V_i \right]
    \cdot D_{ij}
     \Psi_{jk}(u) h_k
    \d x
    \\[\alnspace]
    & \leq
    \Lambda\, L_0(\Psi,M) \,e^{\nm{V}_{\infty}} \,
    \left[ 
    \nm{\nabla h}_{L^2(\Omega)}\,\nm{h}_{L^2(\Omega)}
    \,+\,\nm{\nabla V}_{\infty}\,\nm{h}_{L^2(\Omega)}^2
    \right]
    \\
    & \leq
    \Lambda\, L_0(\Psi,M) \,e^{\nm{V}_{\infty}}\,C_P\,(1+C_P\nm{\nabla V}_{\infty})\nm{\nabla h}_{L^2(\Omega)}^2,
    \\
    & \quad
    \\[\alnspace]
    \jj^{\mathrm{IV}}
    &=\int_{\Omega} 
    \sum_{ijk} e^{V_i}
    \left[ \nabla h_i + h_i\nabla V_i \right]
    \cdot D_{ij}
    (\Psi_{jk}(u) - \Psi_{jk}(\us))\us_{k}
    \d x
    \\[\alnspace]
    & \leq
    \Lambda \,L_1(\Psi,M)\, e^{\nm{V}_{\infty}} \,
    \nm{\us}_{L^\infty(\Omega)}\,
    \left[ 
    \nm{\nabla h}_{L^2(\Omega)}\nm{h}_{L^2(\Omega)}
    +\nm{\nabla V}_{\infty}\nm{h}_{L^2(\Omega)}^2
    \right]
    \\[\alnspace]
    & \leq
    \Lambda\, L_1(\Psi,M) \,e^{\nm{V}_{\infty}}\,
    \nm{\us}_{L^\infty(\Omega)}
    \,C_P\, (1+C_P\nm{\nabla V}_{\infty})\nm{\nabla h}_{L^2(\Omega)}^2.
\end{align*}
Since $\exp(V_i) \geq 1$, combining equations \eqref{eq:testhV} and \eqref{eq:estim1} yields
\begin{align}
    \label{eq:hLtwoestimate}
    \begin{split}
        \ddt \int_{\Omega} \frac{1}{2} |h|^2 \d x
        \,+\,
        &\int_{\Omega} \lambda
        |\nabla h + h\,\diag(\nabla V) |^2
        \d x
        \\[\alnspace]
        &\leq
        \delta
        \Lambda e^{\nm{V}_{\infty}}
        \Gamma_M \,(1+C_P\nm{\nabla V}_{\infty})
        \nm{\nabla h}_{L^2(\Omega)}^2
        \\[\alnspace]
        & \leq
         \delta
        \Lambda e^{\nm{V}_{\infty}}
        \Gamma_M \,
        \frac{1+C_P\nm{\nabla V}_{\infty}}{1-C_P\nm{\nabla V}_{\infty}}
        \int_{\Omega}
        |\nabla h + h\,\diag(\nabla V) |^2
        \d x,
    \end{split}
\end{align}
where we used Remark \ref{rem:P} in the Appendix and the constant $\Gamma_M$ is given by
\begin{align}
    \label{const:gamma}
    \Gamma_M := 
    \max \big\{ &
    L_0(\Phi,M),\;\;
    L_1(\Phi,M)
    \nm{\nabla \us}_{L^\infty(\Omega)}C_P,\;\;
    L_0(\Psi,M) C_P,\;\;
    L_1(\Psi,M) 
    \nm{\us}_{L^\infty(\Omega)}
    C_P
    \big\}.
\end{align}
By assumption  \eqref{eq:deltabound2}, we have the bound
$$
    \delta
    \Lambda e^{\nm{V}_{\infty}}
    \Gamma_M
    \frac{1+C_P\nm{\nabla V}_{\infty}}{1-C_P\nm{\nabla V}_{\infty}} \leq \frac{\lambda}{2}.
$$
This bound in conjunction with estimate \eqref{eq:hLtwoestimate} yields
$$
    \ddt \int_{\Omega}  |h|^2 \d x
    +
    \lambda \int_{\Omega} 
    |\nabla h + h\,\diag(\nabla V) |^2
    \d x
    \leq 0.
$$
Letting $g=h \exp(V)$, we obtain
$$
    \ddt \int_{\Omega}  e^{-2V}|g|^2 \d x
    +
    \lambda \int_{\Omega} 
    e^{-2V}
    |\nabla g |^2
    \d x
    \leq 0.
$$
Using Poincar\'e's inequality for $g$  and integrating in time over the interval $(t_1,t_2)\subseteq [0,\infty)$ yields
$$
    \int_{\Omega}  |g(t_2)|^2 \d x
    +
    \frac{\lambda}{C_P^2} \int_{t_1}^{t_2}\int_{\Omega} 
    | g(\tau) |^2
    \d x \d \tau
    \leq e^{4\nm{V}_\infty} \int_{\Omega}  |g(t_1)|^2 \d x.
$$
Thanks to Gronwall's inequality we obtain the following continuous-dependence estimate
$$
    \int_{\Omega}  |g(t_2)|^2 \d x
    \leq 
    \exp\left(4\nm{V}_\infty- \frac{\lambda}{C_P^2} (t_2 - t_1)\right) \int_{\Omega}  |g(t_1)|^2 \d x.
$$
Finally, let us express this continuous-dependence estimate in terms of the original variables. To this end we recall that $$u - u^* =  h = g \exp(-V).$$ Substituting this relation into the preceding inequality, we deduce that
\begin{align*}
    \nm{u(t_2) - \us}_{L^2(\Omega)}
    &\leq
    \nm{h(t_2) e^V}_{L^2(\Omega)}
    \\[\alnspace]
    &\leq \exp\left(2\nm{V}_\infty- \frac{\lambda}{2C_P^2} (t_2 - t_1)\right) 
    \nm{h(t_1) e^V}_{L^2(\Omega)}
    \\[\alnspace]
    &\leq 
    \exp\left(3\nm{V}_\infty- \frac{\lambda}{2C_P^2} (t_2 - t_1)\right)
    \nm{u(t_1) - \us}_{L^2(\Omega)},
\end{align*}
and this concludes the proof.
\end{proof}

\section{Study of an asymptotic gradient flow for one or two species}\label{sec:agf}
 
In this section we focus on an asymptotic gradient flow system describing the macroscopic dynamics of diffusing species in the presence of volume-exclusion constraints for the agents or particles.
Such systems often have small cross-diffusion terms and therefore fall into the class of problems introduced in Section \ref{sec:model}. We will discuss the equilibration behaviour of their solutions and establish sharper $L^{\infty}$-bounds (recall that, thanks to Proposition \ref{prop:ABC}, we already know that solutions are essentially bounded if $\delta < \delta_0$).

\subsection{A cross-diffusion model for two interacting diffusing  species}\label{sec:agf:twospecies}

We consider the asymptotic gradient flow  system for two species that was introduced by Bruna and Chapman in \cite{Bruna:2012wu}. It describes the behaviour of the densities of two different types of hard spheres and it was derived from a stochastic system of hardcore-interacting Brownian particles using the method of matched asymptotic expansions. Few analytic results are available, and the little results that exist seem to rely on the assumption that the asymptotic gradient flow system is close to its stationary state (see \cite{bruna2017cross}) or that it is close to being a diagonal, decoupled linear problem; cf. \cite{alasio2019global}.

The asymptotic gradient flow system presented in \cite{Bruna:2012wu} can be written in the general form of \eqref{main-general}. In particular:

\begin{equation}
\label{eq:amatrix}
    \amatrix(x,u) = \left(
    \begin{array}{cc}
    	D_{1}&0\\[\alnspace]
    	0 & D_{2}
    \end{array}
    \right)
     \left[I +  
    \left(
    \begin{array}{cc}
    	\delta_{1,1} u_1 - \delta_{1,2} u_2 & \delta_{1,3} u_1\\[\alnspace]
    	\delta_{2,3} u_2 & \delta_{2,1} u_2 - \delta_{2,2} u_1
    \end{array}
    \right)
    \right],
\end{equation}
and
\begin{equation}\label{eq:drift}
	\drift(x, u)=
	\left(
 	\begin{array}{cc}
 		D_1 & 0\\[\alnspace]
 		0 & D_2
 	\end{array}
 	\right)
	\left(
	\begin{array}{cc}
		\nabla V_1 & (\delta_{2,2}\nabla  V_2-\delta_{1,2} \nabla V_1 )u_1\\[\alnspace]
		(\delta_{1,2}\nabla  V_1-\delta_{2,2} \nabla V_2 )u_2  &\nabla V_2
	\end{array}
	\right),
\end{equation}
where $\delta_{i,j}$ for $i\in\{1,2\}, \, j\in \{1,2,3\}$ are asymptotically small parameters with $\max_{i,j} \delta_{i,j}=:\delta <\delta_0$. The coefficients $\delta_{i,j}$ depend on the size of the particles and on the corresponding diffusion coefficients $D_1, D_2\geq 0$, for their specific values we refer to \cite{Bruna:2012wu}.
In \cite{bruna2017cross} it was noted that system \eqref{main-general} with the nonlinearities in \eqref{eq:amatrix} and \eqref{eq:drift} has an \emph{asymptotic gradient flow structure}.
In particular, it is well known that certain cross-diffusion systems possess a formal gradient-flow structure, that is, they can be formulated as
\begin{equation}
	\label{grad_flow_general}
	\partialt u - \nabla\cdot \left ( \mm \nabla \frac{\delta E}{\delta u} \right) = 0, 
\end{equation}
where $\mm \in \mathbb R^{m\times m}$ is usually referred to as the \emph{mobility matrix} and $\delta E/\delta u$ is the \emph{variational derivative} of the entropy functional, $E$. 
In order to highlight the connection between gradient flows and asymptotic gradient flows, let us consider the following entropy functional
\begin{subequations}
	\label{eq:grad_flow1}
\begin{equation}
	E[u] = \int_{\Omega} \bigg [ u_1\log u_1 + u_2\log u_2 + u_1 V_1 + u_2 V_2 
	+\frac{1}{2} \left(  \delta_{1,1} u_1^2 +2(d-1)(\delta_{1,2}+\delta_{2,2}) u_1 u_2 + \delta_{2,1} u_2^2\right)  \bigg ] \d x,\label{eq:entropy_general}
\end{equation}
as well as the mobility matrix 
\begin{equation}
    \mm(u)=\begin{pmatrix}D_1 u_1(1-\delta_{1,2} u_2) & D_1 \delta_{2,2} u_1 u_2\\[\alnspace]
    D_2 \delta_{1,2} u_1 u_2 & D_2 u_2(1- \delta_{2,2} u_1)
    \end{pmatrix}.\label{eq:mobility_general}
\end{equation}
\end{subequations}
The cross-diffusion system \eqref{main-general} with diffusion and drift matrices \eqref{eq:amatrix} and \eqref{eq:drift}, can be rewritten as 
\begin{equation}\label{gradflow_generalasy}
    \partialt u = \nabla\cdot\left ( \mm \nabla \frac{\delta E }{\delta u} -G\right),
\end{equation}

where $G$ is given by 
\begin{equation}
   G= \begin{pmatrix}
 D_1u_1u_2\left(\left(\delta_{1,1}\delta_{1,2}-(d-1)(\delta_{1,2}+\delta_{2,2})\delta_{1,2}\right)\nabla u_1+((d-1)\delta_{1,2}(\delta_{1,2}+\delta_{2,2})-\delta_{2,1}\delta_{2,2} )\nabla u_2 \right) \\[\alnspace]
 D_2u_1u_2\left(\left(\delta_{2,1}\delta_{2,2}-(d-1)(\delta_{1,2}+\delta_{2,2})\delta_{2,2}\right)\nabla u_2+((d-1)\delta_{2,2}(\delta_{1,2}+\delta_{2,2})-\delta_{1,1}\delta_{1,2} )\nabla u_1 \right) 
\end{pmatrix}.
\end{equation}
For a precise definition of all the parameters we refer to the derivation of the model in \cite{Bruna:2012wu}. 

The discrepancy between the system above and 
the gradient-flow induced by \eqref{eq:grad_flow1} is of order $\delta^{2}$, and one can show that they actually coincide if and only if both species have the same diffusivities $D_1 = D_2 \geq 0$ as well as the same particle sizes, cf. \cite{bruna2017cross}. In particular, having  the same size and same diffusivities implies that $\delta_{i,1}=(d-1)(\delta_{1,2}+\delta_{2,2})$ for $i=1,2$ as well as $\delta_{1,j}=\delta_{2,j}$ for $j=1,2,3$, cf. \cite{bruna2017cross} and therefore $G \equiv 0$. Moreover, one can show that this is the only possibility for $G$ to vanish. We refer to \cite{bruna2017cross} for more details as well as the existence proof of unique stationary solutions in both cases. More specifically, we recall the following result:
\begin{prop}
\label{prop:stationary}
    Let $\delta =  \max\{\delta_{i,j}\}$, assume $0<\delta<\delta_0$ and suppose that the potentials in \eqref{main-general} satisfy $V_i\in H^3(\Omega)$. Then Theorem \ref{lem:energysys} applies to systems \eqref{grad_flow_general} and \eqref{gradflow_generalasy} and they admit unique stationary states in $H^3(\Omega)$, which we denote by $u_\infty$ and $u_*$, respectively.
    Additionally, there exists a constant $C>0$ such that
    $$
        \nm{u_* - u_\infty}_{H^3(\Omega)} \leq C\delta^2.
    $$
\end{prop}
As a consequence, we deduce that the stationary states $u_\infty$ and $u_*$ are in the space $W^{1,\infty}(\Omega)$; this fact will be useful in the next subsections.

\subsection{The scalar problem with a ``frozen'' component}

In the following we focus on a special case of system \eqref{gradflow_generalasy}, namely 
\begin{align}
    \label{eq:main-equation}
    \left\{
    \begin{array}{rl}
    \partialt r 
    -\nabla \cdot \left\{(1 + \delta_1 r - \delta_2 b) \nabla r + \delta_3 r \nabla b+r(1-\delta_2 b)\nabla V\right\}\remspace
    &=0
    \quad \text{ in } \Omega,
    \\[\eqnspace]
     \nu \cdot \left\{(1 + \delta_1 r - \delta_2 b) \nabla r + \delta_3 r \nabla b+r(1-\delta_2 b)\nabla V\right\}\remspace
    &= 0
    \quad \text{ on } \de\Omega,
    \\[\eqnspace]
    r(0,x)\remspace{}
    &=
    r_0(x),
    \end{array}
    \right.
\end{align}
where $r=r(t,x)$ describes the density of the mobile species diffusing in the presence of a given immobile species  $b(x)$ (the frozen species).
Notice that we have simplified the notation by letting $\delta_j = \delta_{i,j}$ and $\delta = \max_j \delta_j$ for $i,j=1,2$. 
We chose to use a different notation to emphasise that \eqref{eq:main-equation} is a scalar equation. It can be obtained from system \eqref{gradflow_generalasy} by setting $r=u_1, b=u_2$ and setting the diffusion coefficient as well as the external potential of the species $b$ to zero. 
Note that this equation can be derived as a macroscopic limit of a stochastic system with two types of particles, one species diffusing in a domain with fixed obstacles of a certain size interacting via hard-core collisions. 
Hence, the particles diffuse in a \quot{perforated domain} with obstructions distributed according to the density $b(x)$; cf. \cite{bruna2017asymptotic} for more details. 

\begin{rem}
As equation \eqref{eq:main-equation} is a special case of \eqref{gradflow_generalasy}, namely with immobile blue particles, Proposition \ref{prop:stationary} ensures that
equation \eqref{eq:main-equation} admits a stationary state $r^*\in H^3(\Omega)$ provided that that $\delta$ is sufficiently small, that $b, V \in H^{3}(\Omega)$ and that for some $M_0>0$ it holds that
$$
	0\leq r_0(x)\leq M_0.
$$ 
From the discussion in Section \ref{sec:agf:twospecies} we know that equation \eqref{eq:main-equation} does not exhibit a full gradient flow structure, but only an asymptotic one. Hence, taking the entropy functional 
\begin{align}
    \label{entropy_fp}
    E[r]&=\int_\Omega r\log r +\delta_1 \frac{r^2}{2}+\delta_3 r b +rV\, \d x,
\end{align}
and mobility matrix
$    \mathcal{M}(r)=r(1-\delta_2 b),$
we observe that equation \eqref{eq:main-equation} can be cast in the form \eqref{gradflow_generalasy}
where $G(r,\delta)=rb (\delta_1 \delta_2\nabla r+\delta_2 \delta_3 \nabla b).$
\end{rem}

\begin{rem}
As a consequence of Proposition \ref{prop:ABC}, we have that Problem \ref{eq:main-equation} admits bounded, regular solutions in the space $W$, however, in the rest of this section we will need less regularity for the solutions and we will work in a weaker \quot{parabolic} space, $Z$, that we introduce in Definition \ref{def:parabolic}.
\end{rem}

\begin{defi}[Parabolic space]
\label{def:parabolic}
Consider the space $Z = C^0([0,T]; L^2(\Omega))\cap L^2(0,T;H^1(\Omega))$.
We use the so-called \quot{parabolic norm} 
$$
    \nm{f}_{Z}^2
    = \|f\|_{L^\infty(0,T, L^2(\Omega))}^2 + \|\nabla f\|_{L^2(0,T; L^2(\Omega))}^2.
$$
\end{defi}

Similarly to Theorem \ref{lem:energysys}, we obtain a result on exponential convergence to equilibrium. Note that in the scalar case, we can avoid making the smallness assumption on the $\nabla V(x)$ (see \eqref{eq:deltabound2}), as well as the use of the $L^\infty$-bound on $r$.
\begin{theor}[Exponential convergence to equilibrium]
\label{thm1}
Let $d\in\{1,2,3\}$ and assume that $b(x), V(x) \in H^3(\Omega)$ are such that $V_l\leq V(x)\leq V_u$, for some $V_l,V_u \in \R$ and for all $x\in \Omega$.
Moreover, if $d=1,2$, we suppose that $r\in Z$ is a weak solution of problem \eqref{eq:main-equation} such that $\|r\|_Z \leq L$, for some constant $L>0$; 
if $d=3$ we also suppose that $0\leq r\leq M$ for some constant $M>0$ and a.e. $(t,x)\in[0,T]\times\Omega$. Let the stationary state $r^* \in W^{1,\infty}(\Omega)$.
Then the following inequality holds for any $t_1, t_2 \geq 0$: 
\begin{align*}
	\|r(t_2, \cdot) - r^*\|_{L^2(\Omega)} \leq  \|r(t_1,\cdot) - r^*\|_{L^2(\Omega)}
	\exp\left(
	\frac{3}{2}(V_u - V_l) - \frac{t_2-t_1}{2C_P^2}\right),
\end{align*}
for $\delta$ sufficiently small, satisfying relation \eqref{eq:kdelta} in the proof.
\end{theor}

\begin{proof}
In order to transform the terms in front of the potential $V$ to order $\delta$, we perform the change of variables $w= r\exp(V)$ implying $\nabla r =(\nabla w-w\nabla V) \exp(-V)$ and 
\begin{align}\label{equ_potential}
    e^{-V}\partialt w&=\nabla \cdot \left\{e^{-V}((1+\delta_1 e^{-V} w-\delta_2 b)\nabla w-\delta_1 e^{-V}w^2\nabla V+\delta_3 w\nabla b)\right\}.
\end{align}
Now let us consider a perturbation, $h$, of the stationary state $w^*$ of \eqref{equ_potential} (which exists by Proposition \ref{prop:stationary}), namely 
\begin{align*}
	w(x,t) = w^*(x) + h(x,t),
\end{align*}
where $\int_\Omega h(x,t) = 0$ since the mass is conserved.
We observe that $h$ satisfies the following equation:
\begin{align}\label{equ_linstab}
    e^{-V}\partialt h&= \nabla \cdot \left\{e^{-V} \left[(1+\delta_1 e^{-V}(h+w^*)-\delta_2 b)\, \nabla h\;+\;(\delta_1 e^{-V}\nabla w^*-\delta_1 e^{-V} (2w^*+h)\nabla V +\delta_3  \nabla b) \,h\right]\right\}.
\end{align}
We test equation \eqref{equ_linstab} against $h$ to obtain 
\begin{align}
    \label{equ4}
    \begin{aligned}
        \half \ddt \int_\Omega e^{-V} h^2 \d x 
        &=-\int_\Omega e^{-V}  (1+\delta_1 e^{-V}(h+w^*)-\delta_2 b)|\nabla h|^2\,\d x\\[\alnspace]
        & \quad-\int_\Omega e^{-V}(\delta_1 e^{-V}\nabla w^*-\delta_1 e^{-V} (2w^*+h)\nabla V +\delta_3  \nabla b) h\cdot \nabla h \,\d x.
    \end{aligned}
\end{align}

Integration in time over $(t_1,t_2)\subset [0,T]$ yields

\begin{align}
    \label{equ5}
    \begin{aligned}
        \half\int_\Omega e^{-V} h^2(t_2) \d x 
        &=-\int_{t_1}^{t_2}\int_\Omega e^{-V}  (1+\delta_1 e^{-V}(h+w^*)-\delta_2 b)|\nabla h|^2\,\d x\, \d t\\[\alnspace]
        & \quad-\int_{t_1}^{t_2}\int_\Omega e^{-V}(\delta_1 e^{-V}\nabla w^*-\delta_1 e^{-V} (2w^*+h)\nabla V +\delta_3  \nabla b) h\cdot \nabla h \,\d x
        + \half\int_\Omega e^{-V} h^2(t_1) \d x.
    \end{aligned}
\end{align}

Using the fact that $\nabla V, w^*, \nabla w^*, \nabla b \in L^\infty(\Omega),$ 
we deduce that 
\begin{align}
\label{eq:someintermediateestimate}
\begin{split}
-&\int_{t_1}^{t_2}\int_\Omega e^{-V}
\left(\delta_1 e^{-V}\nabla w^*-2\delta_1 e^{-V} w^*\nabla V +\delta_3  \nabla b\right) h\cdot \nabla h   \d x\, \d t
\\[\alnspace]
&\leq \delta\, e^{-V_l}
\left[
e^{-V_l}\|\nabla w^*\|_{L^\infty(\Omega)}+ 
2e^{-V_l}
\|w^*\|_{L^\infty(\Omega)}
\|\nabla V\|_{L^\infty(\Omega)}
+\|\nabla b\|_{L^\infty(\Omega)}
\right]
\int_{t_1}^{t_2} \|h\|_{L^2(\Omega)}
\|\nabla h\|_{L^2(\Omega)} \d t
\\[\alnspace]
&\leq \delta\, K_1
\int_{t_1}^{t_2}\|\nabla h\|_{L^2(\Omega)}^2 \d t,
\end{split}
\end{align}
where
$K_1 = C_P\left[
e^{-V_l}\|\nabla w^*\|_{L^\infty(\Omega)}+ 
2e^{-V_l}
\|w^*\|_{L^\infty(\Omega)}
\|\nabla V\|_{L^\infty(\Omega)}
+\|\nabla b\|_{L^\infty(\Omega)}
\right]$.
Furthermore, we have that
\begin{align}
\label{eq:someintermediateestimate2}
\begin{split}
\int_{t_1}^{t_2}\int_\Omega 
\delta_1 e^{-2V} h^2 \nabla V \cdot \nabla h   \d x\, \d t 
&\leq 
\delta e^{-2V_l}\|\nabla V\|_{L^\infty(\Omega)}
\int_{t_1}^{t_2}\int_\Omega 
h^2 |\nabla h|   \d x\, \d t
\\[\alnspace]
& \leq
\delta e^{-2V_l}\|\nabla V\|_{L^\infty(\Omega)}
\int_{t_1}^{t_2}\nm{h}_{L^4(\Omega)}^2\nm{\nabla h}_{L^2(\Omega)}\, \d t
\\[\alnspace]
& \leq
\delta e^{-2V_l}\|\nabla V\|_{L^\infty(\Omega)}
C_{GN}\int_{t_1}^{t_2}\nm{h}_{L^2(\Omega)}\nm{\nabla h}_{L^2(\Omega)}^2\, \d t
\\[\alnspace]
& \leq
\delta K_2 \int_{t_1}^{t_2}\nm{\nabla h}^2_{\Omega} \d t,
\end{split}
\end{align}
where, for $d=2$, we set $K_2 = e^{-2V_l}\|\nabla V\|_{L^\infty(\Omega)}
C_{GN} L$ and we used a special case of Gagliardo-Nirenberg interpolation inequality known as Ladyzhenskaya's inequality to bound $\nm{h}_{L^4(\Omega)}$, as well as $\|h\|_{Z}\leq L$; on the other hand, for $d=3$, we set $K_2= 2e^{-2V_l}\|\nabla V\|_{L^\infty(\Omega)}C_S M$ and we used Sobolev's embedding in order to estimate $\nm{h}_{L^4(\Omega)}$.
We define
\begin{equation}
    K(\delta) =   \delta(K_1 + K_2),
\end{equation}
and we substitute equations \eqref{eq:someintermediateestimate} and \eqref{eq:someintermediateestimate2} into equation \eqref{equ5}, which, in turn, becomes
\begin{align}
\half\int_\Omega e^{-V} h^2(t_2) \d x 
+\int_{t_1}^{t_2}\int_\Omega 
e^{-V}  \left[1+\delta_1 e^{-V}(h+w^*)-\delta_2 b - K(\delta)\right]|\nabla h|^2 
\d x \d t
\leq
\half\int_\Omega e^{-V} h^2(t_1) \d x.
\end{align}
Recalling that $h+w^* = w$ is non-negative, we choose $\delta$ sufficiently small so that the following relation is satisfied:
\begin{equation}\label{eq:kdelta}
 K(\delta) < \frac{1}{2}-\delta_2b,
\end{equation}
which implies
\begin{equation}
1+\delta_1 e^{-V}(h+w^*)-\delta_2 b - K(\delta) > \frac{1}{2}.
\end{equation}
Thus we have obtained the following inequality:
\begin{equation}
\int_\Omega e^{-V} h^2(t_2) \d x 
+\int_{t_1}^{t_2}\int_\Omega 
e^{-V}|\nabla h|^2 
\d x \d t
\leq
\int_\Omega e^{-V} h^2(t_1) \d x.
\end{equation}
Using the Poincare's inequality and the fact that $V_l<V<V_u$, we have
\begin{align*}
\int_\Omega h^2(t_2) \d x 
+C_P^{-2} \int_{t_1}^{t_2}\int_\Omega 
|h|^2 
\d x \d t
\leq
e^{V_u - V_l}\int_\Omega  h^2(t_1) \d x.
\end{align*}
Thanks to Gronwall's inequality we obtain the following estimate:
\begin{align*}
	\|h(t_2, \cdot)\|_{L^2(\Omega)} \leq  \|h(t_1, \cdot)\|_{L^2(\Omega)} \exp\left(\frac{1}{2}(V_u-V_l)
	- \frac{t_2-t_1}{2C_P^2}\right).
\end{align*}
Finally, let us switch back to the original variables by recalling the fact that $h = w-w^*$. We obtain
\begin{align*}
	\|w(t_2, \cdot) - w^*\|_{L^2(\Omega)} \leq  \|w(t_1,\cdot) - w^*\|_{L^2(\Omega)} \exp\left(\frac{1}{2}(V_u-V_l)
- \frac{t_2-t_1}{2C_P^2}\right).
\end{align*}
Since $V$ is bounded from above and below by assumption, a simple change of variable $r = w e^{-V}$ implies the exponential convergence to equilibrium for the variable $r$ as claimed in the statement, which concludes the proof.
\end{proof}

\subsection{A maximum principle for the scalar problem ($d=2$)}

In this section we will improve the $L^\infty$-bounds for solutions of problems of type \eqref{eq:main-equation}, in dimension $d=2$ (the same result can be obtained in one dimension).
We adapt the strategy presented in \cite{ladyzhenskaia1988linear}, Chapter V\footnote{Namely, Chapter V, Section 2, p. 425, Theorem 2.1; Chapter II, Section 6, p. 102, Theorem 6.1 (see also Lemma 5.6 on p. 95).}.
In the context of elliptic equations, this approach is referred to as De Giorgi's method (see, e.g., \cite{han2011elliptic}).

\begin{theor}[Maximum principle]
\label{theor:maximum}
Let $d=2$ and suppose that $\delta$ is sufficiently small and let $r\in Z$ be a solution of problem \eqref{eq:main-equation}. Then we have that  
\begin{equation}\label{eq:max-princ}
    r(t,x) \leq (1+C\delta^{\sigma})\,
    \exp\left(\nm{V}_{L^\infty(\Omega)} - V(x)\right)\,
    \nm{r_0}_{L^\infty(\Omega)},
\end{equation}
 for a.e. $t\in [0,\infty)$, $x\in\Omega$ and for some $\sigma\in(\frac{1}{2}, 1)$. The constant $C>0$ does not depend on time.
\end{theor}

The proof of Theorem \ref{theor:maximum} uses several technical results which we shall recall in the following.
Notice that, if $x$ is close to a maximum point of $V$, estimate \eqref{eq:max-princ} is \quot{almost sharp}, in the sense that $r$ is estimated by $ \nm{r_0}_{L^\infty(\Omega)}$ with a multiplicative constant that is very close to $1$.

\ms

The following result is essential in order to prove Theorem \ref{theor:maximum}.

\begin{lem}\label{lem:maxiimum}
Let $d=2$ and suppose that $z\in Z$ is a weak solution of the problem

\begin{align}
    \label{eq:generalscalar}
    \left\{
    \begin{array}{rl}
    \omega(x) \partialt z -\nabla \cdot \left\{\omega(x) ( (1+\delta A(x,z))\nabla z + \delta F(x,z) ) \right\} \remspace &= 0  \quad \text{ in } \Omega,\\[\eqnspace]
    \nu \cdot \left\{(1+\delta A(x,z) \nabla z + \delta F(x,z))\right\} \remspace &= 0  \quad \text{ on } \de\Omega, \\[\eqnspace]
    z(0,x)\remspace  &= z_0(x).
    \end{array}
    \right.
\end{align}
We assume that $\delta\in(0,1)$ is sufficiently small (see \eqref{eq:deltabound1} and the assumptions below) and, furthermore, we suppose that
\begin{enumerate}[(I)]
    \item $\omega:\RR^d\to\RR$ satisfies $0 < \mu \leq \omega(x) < \mu^{-1}$, for a.e. $x\in\Omega$,\label{en:UpperLowerOnOmega}\\[-0.9em]
    \item $A:\RR^d\times\RR\to\RR$ and $\delta$ satisfy $0<\lambda \leq 1 + \delta A(x,z)$, for a.e. $x\in\Omega$,\\[-0.9em]
    \item $F(x,z) = z F_1(x) + z^2 F_2(x)$, where $F_i\in L^{\infty}(\Omega)$, and $i=1,2$,\\[-0.9em]
    \item there exists $\bar{M}>0$ such that $0 \leq z_0(x)\leq \bar{M}$, for a.e. $x\in\Omega$,\\[-0.9em]
    \item there exists a constant $L>0$ depending only on $z_0, \Omega, F, \alpha, \omega, \delta$ such that 
    $\nm{ z }_{Z} \leq L$.
\end{enumerate}
Then there exists a constant $c>0$, not depending on time and given in \eqref{eq:estmax}, such that 
$$
    \nm{z}_{L^\infty(\Omega)}\leq c \bar{M}.
$$
\end{lem}
Before proving Lemma \ref{lem:maxiimum}, we state a preliminary result from Chapter II in \cite{ladyzhenskaia1988linear}. It can be proved directly by induction.
\begin{lem}\label{lem:sequence}
Suppose that a sequence $(a_n)_{n\in\NN} \subset \R_+$, of non-negative numbers satisfies the recursive relation
$$
a_{n+1} \leq \kappa \, \zeta^n \, a_n^{1+\eps}
\quad \text{ and } \quad
a_0 \leq \kappa^{-\frac{1}{\eps}} \zeta^{-\frac{1}{\eps^2}},
$$
for $\kappa, \eps > 0$ and $\zeta>1$.
Then we have that
$$
a_n\leq \kappa^{-\frac{1}{\eps}} \, \zeta^{-\frac{1}{\eps^2}-\frac{n}{\eps}}.
$$
\end{lem}
We will use the result above in order to conclude the following proof.
\begin{proof}[Proof of Lemma \ref{lem:maxiimum}]
The proof consists of several steps. We begin with an energy estimate involving Stampacchia's truncation and continue with an estimate for the measure of the superlevel sets 
$$
    S_{z>k}(t):=\{x\in \Omega \, : \, z(x,t) > k \},
$$
for $k\in \RR_+$ and $t>0$. 
We will often use the abbreviation $\{z > k \}$ instead of $\{x\in \Omega \, : \, z(x,t) > k \}$.
The strategy of the proof is the following: via careful use of a priori estimates, we will construct a sequence of the form
$$
a_n = \left(
    \int_0^T|S_{z>M(2-2^{-n})}(t)| \d t
    \right)^q,
$$
for $n\in\NN$ and  suitable $M>0$ as well as $q>0$. Subsequently, we will apply Lemma \ref{lem:sequence} to $a_n$ in order to deduce that
$|S_{z>M}(t)| = 0$ for a.e. $t\in[0,T]$. This implies that $z$ is essentially bounded.

\textbf{Step 0: Stampacchia's truncation}\\
Let $\tilde{M}\geq \bar{M}$ (to be chosen later) and $k\in\RR_+$, $k>\max(\tilde{M}, 1)$. 
It is easily verified that  $(z-k)_+ \in Z$ (see, e.g. Theorem 2.1.11 in \cite{ziemer2012weakly} or Lemma 4.4 in Chapter II of \cite{ladyzhenskaia1988linear}). Thus we may test equation \eqref{eq:generalscalar} against the truncated function $(z-k)_+$ and obtain
$$
    \ddt \int_\Omega
    \frac{1}{2}\omega(x)  (z-k)_+^2 
    \d x
    =
    -\int_\Omega
     [ \omega(x) (1+\delta  A(x,z))|\nabla (z-k)_+|^2 + \delta \omega(x) F(x,z)\cdot\nabla (z-k)_+ ]
     \d x.
$$
Rearranging this relation we get
\begin{equation}
\label{eq:stamptrunc}
    \ddt \int_\Omega
    \frac{1}{2}\omega(x)  (z-k)_+^2 
    \d x
    +
    \int_\Omega
    \omega(x) \lambda |\nabla (z-k)_+|^2
       \d x
       \leq
     \frac{\delta}{\mu}  
     \left|  \int_\Omega   
     F(x,z)\cdot\nabla (z-k)_+ 
     \d x  \right|.
\end{equation}

\textbf{Step 1: upper bound for the RHS in terms of $|\{z\geq k\}|$}\\
Let us focus on the estimate of the right-hand side:
\begin{align*}
     \int_\Omega   
     F(x,z)\cdot\nabla (z-k)_+ 
     \d x
     &=
        \int_\Omega   
        ( z F_1(x) + z^2 F_2(x) )\cdot\nabla (z-k)_+
        ) \d x \\[\alnspace]
     & \leq
         \nm{F_1}_{L^\infty(\Omega)} 
         \int_\Omega
         ((z-k)_+ + k)|\nabla (z-k)_+| 
          \d x\\[\alnspace]
     & \qquad
        +   
        \nm{F_2}_{L^\infty(\Omega)}
        \int_\Omega
        ((z-k)_+ + k)^2 |\nabla (z-k)_+|
        \d x\\[\alnspace]
     & \leq
        C_{V}
       \int_\Omega
       \left\{
       (z-k)_+^2|\nabla (z-k)_+|
       + 3k(z-k)_+|\nabla (z-k)_+|
       \right\}
        \d x\\[\alnspace]
    &\qquad    
        + C_{V}
       \int_\Omega (k^2+k)|\nabla (z-k)_+|   \d x\\[\alnspace]
    & \leq
        \frac{5}{2} C_{V}
        \int_\Omega
        (z-k)_+^2|\nabla (z-k)_+|
        \d x
        + \frac{7}{2} C_{V}
       \int_\Omega k^2|\nabla (z-k)_+|   \d x,
\end{align*}
where we set $C_{V}= \max\{ \nm{ F_1}_{L^\infty(\Omega)}, \nm{ F_2}_{L^\infty(\Omega)} \}$ and used the fact that $k\geq 1$.\\

Notice that the integral on the right-hand side of equation \eqref{eq:stamptrunc} may be estimated further by
\begin{align}
    \label{eq:13180706}
    \begin{split}
         \int_\Omega   
         F(x,z)\cdot\nabla (z-k)_+ 
         \d x
         & \leq
           \frac{C_{V}}{2}
            \int_\Omega
            \left\{
           5(z-k)_+^2|\nabla (z-k)_+|
           +
        7 k^2|\nabla (z-k)_+| 
        \right\} \d x
        \\[\alnspace]
         & \leq
           \frac{C_{V}}{2}
            \left\{
           5\nm{(z-k)_+}_{L^4(\Omega)}^2\nm{\nabla (z-k)_+}_{L^2(\Omega)}
           +
        7 k^2\nm{\nabla (z-k)_+}_{L^1(\Omega)} 
        \right\}.
    \end{split}
\end{align}
Since $z$ is a Lebesgue-integrable function, it cannot be infinite on a region of full measure, therefore, for a sufficiently large $k$ (depending only on $\int_\Omega z_0 \d z$, thanks to Chebyshev's inequality),  the function $(z-k)_+$ vanishes on a region of positive measure. 
This implies that the modified version of Poincar\'e's inequality given in Lemma \ref{lem:poincare} holds, provided we can show that the Poincar\'e constant $C_P$ is uniformly bounded for all possible solutions $z$ (we do not assume uniqueness).
In order to control the Poincar\'e constant it is essential to obtain a lower bound on the measure of the set $\{x \,\big| \, (z-k)_+ = 0\}$.
To do so, let us make the following remark:

\begin{rem}\label{rem:cheb}
Let $f\in L^1(\Omega)$ be any non-negative function such that $\int_\Omega f \d x = m$, for some positive constant $m$ depending only on $\Omega$. Recall that, by Chebyshev's inequality, for any $q>0$ it holds that
$$
|\{ x \,\big|\, f \geq q, \text{ a.e.}\}|
\leq \frac{1}{q} \int_\Omega f \d x.
$$
Then we have, for $q=\frac{m}{\alpha |\Omega|}$ and an arbitrary $\alpha\in(0,1)$,
$$
    \left| \left\{ x \,\big|\, f(x) < \frac{m}{\alpha |\Omega|} 
    \right\} \right|
    \,=\, |\Omega| - 
    \left| \left\{ x \,\big|\, f(x) \geq  \frac{m}{\alpha |\Omega|}
    \right\} \right|
    \,\geq\, |\Omega| - \frac{\alpha |\Omega|}{m} \int_\Omega f \d x
    \,=\, |\Omega|(1-\alpha).
$$
\end{rem}

We recall that integrating \eqref{eq:generalscalar} over the whole domain gives $\int_\Omega \omega z \d x = \int_\Omega \omega z_0 \d x$, for any $t>0$, as $z$ is continuous with respect to time.
We can now argue as in Remark \ref{rem:cheb} with $m=\int_\Omega \omega z_0 \d x$, $f=\omega z$, $\alpha \in(0,1)$ and $k \geq \frac{1}{\alpha \mu|\Omega|}\int_\Omega \omega z_0 \d x$, to the set $\{x \,\big|\, \omega z \leq \mu k\}$, which, by assumption (\ref{en:UpperLowerOnOmega}), is contained in the set we are interested in, i.e.,
$\{x \,\big|\, z(t,x) \leq k \}$.
Consequently, the Poincar\'e constant $C_P$ given in the second part of Lemma \ref{lem:poincare} (with $\sigma_0 = 1-\alpha$) is always bounded independently of $z$.

Since $d=2$, we derive a bound for $\nm{(z-k)_+}_{L^4(\Omega)}^2$ applying Gagliardo-Nirenberg and Poincar\'e's inequalities (see Lemma \ref{lem:poincare}); namely we get:
\begin{align}\label{eq:Lineq}
     \nm{(z-k)_+}_{L^4(\Omega)}^2 
    & \leq 
    C_{GN}
    \nm{(z-k)_+}_{L^2(\Omega)} \nm{(z-k)_+}_{H^1(\Omega)}
    \\[\alnspace]
    & \leq 
    C_{GN}
    \nm{(z-k)_+}_{L^2(\Omega)} \left(
    \nm{\nabla(z-k)_+}_{L^2(\Omega)}
    +
    \nm{(z-k)_+}_{L^2(\Omega)}
    \right)
    \nonumber \\[\alnspace]
    & \leq
      C_{GN} (1+C_P)  \nm{(z-k)_+}_{L^2(\Omega)} \nm{\nabla(z-k)_+}_{L^2(\Omega)}, 
      \nonumber 
\end{align}
This may be further estimated upon observing that 
$$
    \int_\Omega |z|^2\d x \geq \int_{\{z>k\}} (z-k+k)^2\d x
    =
    \int_{\{z>k\}}\left( 
    (z-k)^2 +2k(z-k) + k^2
    \right)\d x
    \geq \nm{(z-k)_+}_{L^2(\Omega)}^2.
$$
Using this in conjunction with assumption (5) in the statement we obtain $\nm{(z-k)_+}_{L^2(\Omega)}\leq \nm{z}_{Z}  \leq L$.

Then the previous inequality leads to the following estimate:
\begin{align*}
     \nm{(z-k)_+}_{L^4(\Omega)}^2 \leq  C_{GN} (1+C_P) L  \nm{\nabla(z-k)_+}_{L^2(\Omega)}.
\end{align*}

Notice that for the remaining term in estimate \eqref{eq:13180706}, we have 
$$
    k^2 \int_\Omega|\nabla (z-k)_+| \d x
    \leq 
    \frac{k^2}{2}  \left(|\{z\geq k\}| + \int_\Omega|\nabla (z-k)_+|^2 \d x\right).
$$
Hence estimate \eqref{eq:13180706} becomes
$$
     \int_\Omega   
     F(x,z)\cdot\nabla (z-k)_+ 
     \d x
     \leq
       \frac{C_{V}}{2} 
        \left\{
       \left[ 5 C_{GN} (1+C_P)L +\frac72 k^2\right]
       \nm{\nabla(z-k)_+}_{L^2(\Omega)}^2
       +
    \frac72 k^2 |\{z\geq k\}|
    \right\}.
$$

Altogether we have obtained the inequality:
$$
    \ddt \int_\Omega
    \frac{1}{2}\omega(x)  (z-k)_+^2 
    \d x
    +
    \int_\Omega
     (\omega(x)\lambda - \delta C_F )|\nabla (z-k)_+|^2
       \d x
       \leq
    \frac{7\delta C_V}{4\mu} k^2 |\{z\geq k\}|,
$$
where 
$C_F =  \frac{C_V}{2\mu} \left[5L  C_{GN} (1+C_P) + \frac{7}{2}M^2 \right]$. 
Since $\omega \geq \mu$ and $k\geq \bar{M}$, after an integration in time, we obtain
$$
    \frac{\mu}{2}\max_{t\in[0,T]} \int_\Omega
    (z-k)_+^2 
    \d x
    +
    \int_0^T \int_\Omega
    (\mu\lambda - \delta C_F )|\nabla (z-k)_+|^2
       \d x \d t
       \leq
    \frac{7\delta C_V}{4\mu} k^2 \int_0^T|\{z\geq k\}|\d t.
$$
Supposing that $\delta$ is such that 
\begin{equation}
    \label{eq:deltabound1}
    \mu \lambda - \delta C_F >0,
\end{equation}
we have
\begin{align}
    \label{eq:12061640}
    \nm{ (z-k)_+ }_Z^2
       \leq
    \delta \Gamma  k^2 \int_0^T|\{z\geq k\}| \d t,
\end{align}
where $\Gamma = 7C_V(4\mu^2\min\{\mu/2,\; \mu\lambda - \delta C_F   \} )^{-1}.$\\

\textbf{Step 2: lower bound for the LHS in terms of $|\{z\geq k\}|$}\\
Notice that, as above, there exists a constant $C_e>0$ such that the following inequality holds for any $1\leq p < \infty$ (it is a direct consequence of Sobolev's and Poincar\'e's inequality):
$$
    \nm{f}_{L^2(0,T;L^p(\Omega))} \leq C_e \nm{ f }_Z.
$$
We also know that, for any $f\in Z$, and thus in particular for $f=(z-k)_+$, there holds $\nm{f}_{L^\infty(0,T;L^2(\Omega))} \leq  \nm{ f }_Z$. 
Thanks to the Riesz--Thorin interpolation theorem, we deduce that, for $\nu=p^{-1} \in (0,1)$, 
\begin{align}
    \label{eq:finL4nu}
    \nm{f}_{L^{4-\nu}(Q_T)} < \max\{1,C_e\}\nm{ f }_Z,
\end{align}
where $C_e$ does not depend on time. 
By Chebyshev's inequality, for any $a\geq 1$, we obtain 
$$
    |\{ (z-h)_+ > 0 \}| = |\{ z > h \}| = |\{ z-k > h-k \}| 
    \leq 
    \frac{1}{(h-k)^a} 
    \int_{\{ z > h \}}  (z-k)_+^a \d x,
$$
whence, upon integration, we obtain
$$
    \int_0^T|\{ z > h \}| \d t
    \leq
    \frac{1}{(h-k)^a} 
    \int_0^T\int_\Omega  (z-k)_+^a \d x \d t.
$$
Choosing $a=4-\nu$, we get
\begin{align}
    \label{eq:1206191631}
    \left(
    \int_0^T|\{ z > h \}| \d t
    \right)^{\frac{2}{4-\nu}}
    \leq
    \frac{1}{(h-k)^2} 
    \left(
    \int_0^T\int_\Omega  (z-k)_+^{4-\nu} \d x \d t\right)^{\frac{2}{4-\nu}}.
\end{align}
\textbf{Step 3: combining all previous estimates}\\
Combining \eqref{eq:12061640}, \eqref{eq:finL4nu}, and \eqref{eq:1206191631} we obtain
\begin{align}
    \label{eq:12061615}
    \begin{split}
    \left(
\int_0^T|\{ z > h \}| \d t
\right)^{\frac{2}{4-\nu}}
    &\leq
    \frac{1}{(h-k)^2} \max\{1,C_e^2\}
    \nm{ (z-k)_+ }_Z^2 \\
    &\leq  \frac{ k^2}{(h-k)^2} {\tilde{\Gamma}_\delta}  \int_0^T|\{z\geq k\}| \d t,
    \end{split}
\end{align}
where
${\tilde{\Gamma}_\delta} =  \max\{1,C_e^2\}   \delta \Gamma$. For $M > \bar{M}$, consider the increasing sequence defined by
$$
    k_n = M(2-2^{-n}),
$$
and
$$
a_n = \left(
    \int_0^T|\{ z > k_n \}| \d t
    \right)^{\frac{2}{4-\nu}}.
$$
for any  $n\in\NN$. In order to apply Lemma \ref{lem:sequence} we need to define a recursive relation for the sequence $(a_n)_{n\in \NN}$. Let us set 
$$
    k= k_{n}, \quad h = k_{n+1}, 
$$
then we observe that
\begin{align*}
    \frac{k^2}{(h - k)^2} = \frac{k_{n}^2}{(k_{n+1} - k_{n})^2} = \left(\frac{2 -2^{-n}}{2^{-n} - 2^{-n-1}}\right)^2
    =\left( 2^{(n+1)}(2-2^{-n}) \right)^2 \leq 4^{n+2},
\end{align*}
we observe that, from \eqref{eq:12061615},
\begin{align*}
    a_{n+1} = \left(
    \int_0^T|\{ z > k_{n+1} \}| \d t
    \right)^{\frac{2}{4-\nu}} \leq 4^{n+2} \tilde{\Gamma}_\delta \int_0^T |\{z\geq k_{n}\}|\d t =  4^{n+2} \tilde{\Gamma}_\delta a_n^{\frac{4-\nu}{2}}.
\end{align*}
Thus the sequence $(a_n)_{n\in \NN}$ satisfies the following relation

\begin{align*}
    a_{n+1} \leq   4^{n+2} \tilde{\Gamma}_\delta a_n^{\frac{4-\nu}{2}}.
\end{align*}
Let us set $\varepsilon = 1- \nu/2$, $\zeta=4$, and $\kappa=16 \tilde{\Gamma}_\delta$ and apply Lemma \ref{lem:sequence}. We deduce that $a_n\to 0$ as $n\to\infty$,  and that the following inequality holds:
$$
    a_n\leq {(16\tilde{\Gamma}_\delta)} ^{-\frac{1}{\eps}} 4^{-\frac{1}{\eps^2}-\frac{n}{\eps}},
$$
provided that
$$
    a_0 = \left(
    \int_0^T|\{ z > \tilde{M} \}| \d t
    \right)^{\frac{2}{4-\nu}} \leq 
    {(16\tilde{\Gamma}_\delta)}^{-\frac{1}{\eps}} 4^{-\frac{1}{\eps^2}}.
$$
This will be satisfied choosing $\tilde{M}$ sufficiently large. Thus, we shall now concentrate on giving an explicit lower bound for $\tilde{M}$. To this end, let $\theta >1$ be such that $\tilde{M} = \theta \bar{M}$. From \eqref{eq:finL4nu} and \eqref{eq:1206191631} and upon choosing $h=\theta \bar{M}$ and $k= \bar{M}$, we also know that, 
\begin{align*}
    \left(
    \int_0^T|\{ z > \theta \bar{M} \}| \d t
    \right)^{\frac{2}{4-\nu}}
    &\leq
    \frac{\max\{1,C_e^2\}}{\bar{M}^2 (\theta - 1)^2} 
    \| (z-\bar{M})_+ \|_{Z}^2.
\end{align*}
Notice that the quantity $\nm{ (z-\bar{M})_+ }_Z$ is bounded uniformly in time by $ C_0 = L + \bar{M}$.

Therefore we impose that $\theta$ satisfies
$$
    \frac{ C_0^2  }{\bar{M}^2(\theta-1)^2} \leq{(16\tilde{\Gamma}_\delta)}^{-\frac{1}{\eps}} 4^{-\frac{1}{\eps^2}},
$$
which, in turn, gives
$$
    \theta \geq 1 + \frac{ \bar{M}  }{C_0} 
    \sqrt{ 
        {(16\tilde{\Gamma}_\delta)}^{-\frac{1}{\eps}} 4^{-\frac{1}{\eps^2}} 
    },
$$
such that the upper bound for $a_0$ is met. 
We conclude that 
$$
    \int_0^T|\{ z > \theta \bar{M} \}| \d t = 0,
$$
after passing $n\to \infty$, provided that
$$
    \tilde{M}\geq \bar{M} \left(
    1 + \frac{ \bar{M}  }{C_0} 
    \sqrt{ 
        {(16\tilde{\Gamma}_\delta)}^{-\frac{1}{\eps}} 4^{-\frac{1}{\eps^2}} 
    },  \right),
$$
and consequently we have obtained that, for a.e. $t,x$,
\begin{equation}\label{eq:estmax}
    z\leq \left( 1 + \frac{ \bar{M}  }{C_0}  {\tilde{\Gamma}_\delta} ^{\frac{1}{2\eps}} 4^{\frac{1}{\eps^2}}  \right) \bar{M} 
    = c\bar{M}.
\end{equation}
\end{proof}
\begin{rem}[Higher dimensions]
\label{rem:3d}
Lemma \ref{lem:maxiimum} does not hold in general for $d > 2$. For example, estimate \eqref{eq:Lineq} is not valid in the same form if $d=3$.
The main difficulty consists in finding bounds for the term $\delta \nabla \cdot(\omega(x) z(t,x)^2 F_2(x))$ in equation \eqref{eq:generalscalar}.
Sufficient conditions for $L^\infty$ estimates are presented in \cite{ladyzhenskaia1988linear}, Chap. V, Sec. 2, p. 423, equation (2.3), and they impose restrictions on the growth rates of the flux terms. Such conditions could not be verified in our case.
\end{rem}
With Lemma \ref{lem:maxiimum} at hand, the proof of Theorem \ref{theor:maximum} is immediate. 

\begin{proof}[Proof of Theorem \ref{theor:maximum}.]
We apply Lemma \ref{lem:maxiimum} making the following choices:
\begin{align*}
    & z=r \exp(V),
    && \omega = \exp(-V),
    \\
    & A(\cdot,r) = r-b,
    && L=\Gamma_0 \text{ (see Proposition \ref{prop:ABC})},
    \\
    & F_1 = \nabla(b + V),
    && F_2 = \nabla V.
\end{align*}
Recall that rewriting equation \eqref{eq:main-equation} in terms of the new unknown $z = u \exp(V)$ we obtain
$$
    e^{-V}\partialt z
    = \nabla \cdot \left\{e^{-V}((1+\delta e^{-V} z-\delta b)\nabla z-\delta (e^{-V}z^2\nabla V+z\nabla b))\right\},
$$
with suitable initial and boundary conditions.
\end{proof}

\section{Numerical Simulations for the Scalar Problem}\label{sec:numerics}

Let us recall the equation
\begin{align*}
	\frac{\partial r}{\partial t} = \nabla \cdot ((1 + \delta_1 r - \delta_2 b) \nabla r + \delta_3 r \nabla b + r (1-\delta_2 b)\nabla V),
\end{align*}
with zero flux conditions at the boundary. We note that the equation can be cast into a transport part and a remainder part, i.e.,
\begin{align*}
	\frac{\partial r}{\partial t}  = \nabla \cdot \left(r [\nabla \left(\log(r) + \delta_1 r + \delta_3 b\right) + (1-\delta_2 b)\nabla V]   - \delta_2 b \nabla r \right).
\end{align*}
This formulation is the basis of the finite volume scheme used for the following section. The scheme is based on the schemes studied in \cite{carrillo2018fvconvergence, CHS18, bessemoulin2012finite}. For simplicity we introduce it here in one space dimension but an extension to two (or more) dimensions is immediate. We note that we could have used a different formulation to define the numerical fluxes, for example, based on \eqref{grad_flow_general} with entropy \eqref{entropy_fp}. However, the formulation chosen by us, is somewhat more natural, since it clearly highlights the transport part. \\

In order to discretise the spatial domain, $\Omega = (-L,L)$, for some $L>0$, we introduce the computational mesh consisting of the control volumes $C_i=[x_\imhalf, x_\ihalf)$, for all $i \in I:=\{1,\ldots, N\}$. The measure of each control volume is given by $|C_i| = \Delta x_i = x_{i+1/2} - x_{i-1/2}>0$, for all $i\in I$. Note that $x_{1/2}=-L$, and $x_{N+1/2}=L$. We also define $x_i = (x_{i+1/2} + x_{i-1/2})/2$ the centre of cell $C_i$ and set $\Delta x_{i+1/2}=x_{i+1}-x_i$ for  $i=1,\ldots,N-1$. 

Next, we discretise the initial data by computing the cell averages of the continuous initial data on each cell, i.e.,
\begin{align*}
	r_i^0 := \frac{1}{\Delta x_i}\int_{C_i} r_{\mathrm{init}}(x)\,\d x,
\end{align*}
for all $i\in I$. Upon integrating the equation over $[t^n, t^{n+1}) \times C_i$, we obtain the following finite volume approximation
\label{eq:scheme}
\begin{align}
	\label{eq:scheme_evol}
	\frac{r_i^{n+1} - r_i^n}{\Delta t} = - \frac{\mF_{\ihalf}^n - \mF_{i-1/2}^n}{\Delta x_i},
\end{align}
for $i\in I$. Here the numerical fluxes are given by
\begin{align*}
	\label{eq:numerical_fluxes}
	\mF_{\ihalf}^n &= r_i^n \left[( -\d \xi)_{\ihalf,+} + (1 -\delta_2 \,b_{\ihalf}\,) \left(-\frac{V_{i+1} - V_i}{\Delta x_\ihalf}\right)_+\right]\\
	&\quad + r_{i+1}^{n} \left[(-\d \xi)_{\ihalf,-} + (1 -\delta_2 \,b_{\ihalf}\,) \left(-\frac{V_{i+1} - V_i}{\Delta x_\ihalf}\right)_-\right]\\
	&\quad + \delta_2 \,b_{\ihalf}\, \frac{r_{i+1}^n - r_i^n}{\Delta x_{\ihalf}},
\end{align*}
for $i = 1,\ldots, N-1$, and 
\begin{align*}
	b_\ihalf = b(x_\ihalf),	
\end{align*}
 as well as
\begin{align*}
	\xi_i^n := \log(r_i^n + \varepsilon) \,+\, \delta_1 \, r_i^n \,+\, \delta_3 \, b(x_i),
\end{align*}
where $0< \varepsilon = 10^{-7} \ll 1$ is a small constant that is commonly chosen to regularise the logarithm. In addition, we impose the numerical no-flux boundary conditions $\mF_{1/2} = \mF_{N+1/2} =0$. As usual, we use $(z)_\pm$ to denote the positive (resp. negative) part of $z$, i.e.,
\begin{align*}
	(z)_+ := \mathrm{max}(z,0), \qquad \text{and} \qquad (z)_-:=\mathrm{min}(z,0).
\end{align*}

\subsection{One Dimensional Explorations}

Let us fix $\Omega = [-5,5]$ as computational domain and set $\delta_1 = 10^{-1}$ in this subsection. The immobile species is given by $b(x) = 1/\sqrt{2\sigma^2 \pi}  \exp(- x^2/(2\sigma^2))$, with $\sigma = 10^{-1}$. The active species is initially distributed according to $r_0(x) = 1/\sqrt{2 \sigma^2 \pi}  \exp(-(x + 3)^2 / (2\sigma^2))$, with $\sigma = \sqrt{1/2}$. The external potential is chosen as 
$V(x)= (7/2 - x)^2 \chi_{\{x < 7/2\}}(x) + (x - 7/2)^5 \chi_{\{x\geq 7/2\}}(x)$; cf. Figure \ref{fig:explanation_choices}. This choice forces the mobile species to penetrate the immobile species as it migrates towards the minimum of the external potential. In Figure \ref{fig:voleffects_inhibition} we present the evolution of the mobile species for three different choices of $\delta_2 = \delta_3$. It is apparent that an increase in the parameters $\delta_2= \delta_3$ leads to a more and more inhibited migration. This behaviour is perfectly physical. In fact, the parameters $\delta_i$ are directly linked to the radii of particles of the mobile (resp. immobile) species. Since it is harder for large particles to traverse other large particles the motion is increasingly slowed down.
For the second numerical exploration we choose $\tilde V(x)=x^2/2$ and assume that the mobile species is initially distributed according to $\tilde r_0(x) = \frac12 r_0(x-3) + \frac12 r_0(x+3)$, with $r_0$ as above; cf. Figure \ref{fig:explanation_choices}. Figure \ref{fig:voleffects_inhibition2} shows the migration of the two initial bumps of the mobile species towards the minimum of the external potential. Since the immobile species inhabits this region we observe a competition for space leading to an indentation in the density of the mobile species, showcasing the strong effect of finite size exclusion effects. In Figure \ref{fig:expconvoneD} we show the exponential convergence of the solution towards equilibrium.
\begin{figure}[ht!]
    \centering
    \subfigure[Initial data $r_0$ and $\tilde r_0$.]{
    \includegraphics[width=0.3\textwidth]{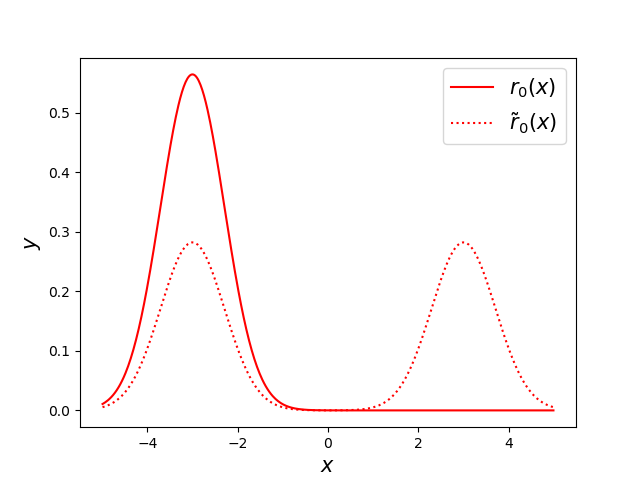}
    }
    \subfigure[External potentials $V$ and $\tilde V$.]{
    \includegraphics[width=0.3\textwidth]{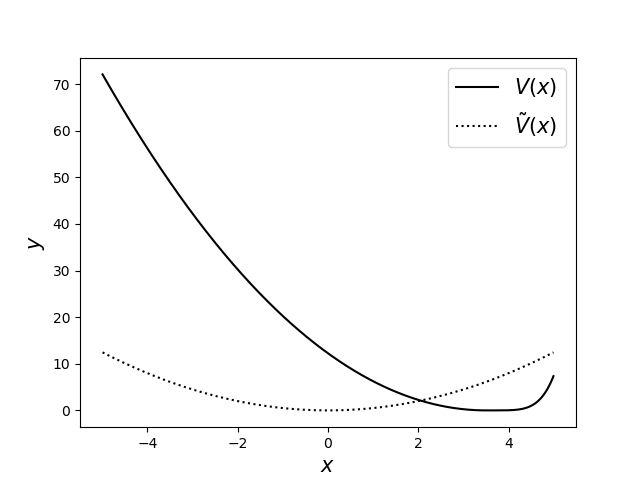}
    }
    \subfigure[Immobile species $b$.]{
    \includegraphics[width=0.3\textwidth]{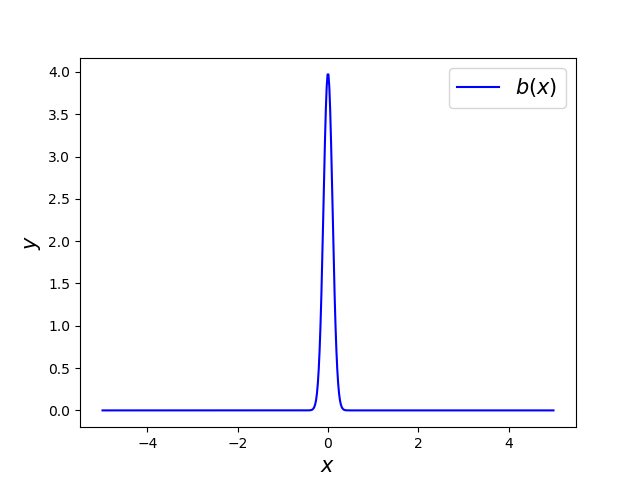}
    }
    \caption{We provide two simulations based on the same set of parameters for two types of initial data (left), external potential (centre), and a given immobile species (right).}
    \label{fig:explanation_choices}
\end{figure}

\begin{figure}[ht!]
    \centering
    \subfigure[$\delta_2= \delta_3 = 1 \times 10^{-1}$.]{
    \includegraphics[width=0.3\textwidth]{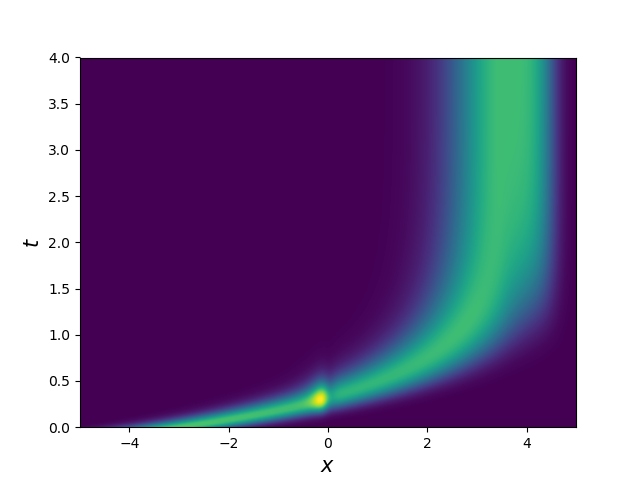}
    }
    \subfigure[$\delta_2= \delta_3 = 2 \times 10^{-1}$.]{
    \includegraphics[width=0.3\textwidth]{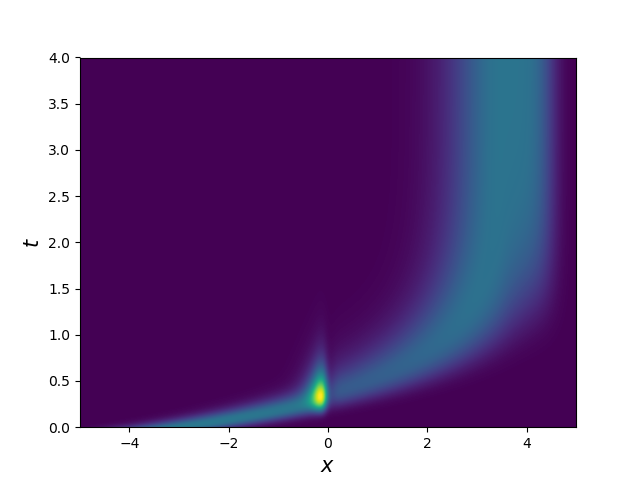}
    }
    \subfigure[$\delta_2= \delta_3 = 3 \times 10^{-1}$.]{
    \includegraphics[width=0.3\textwidth]{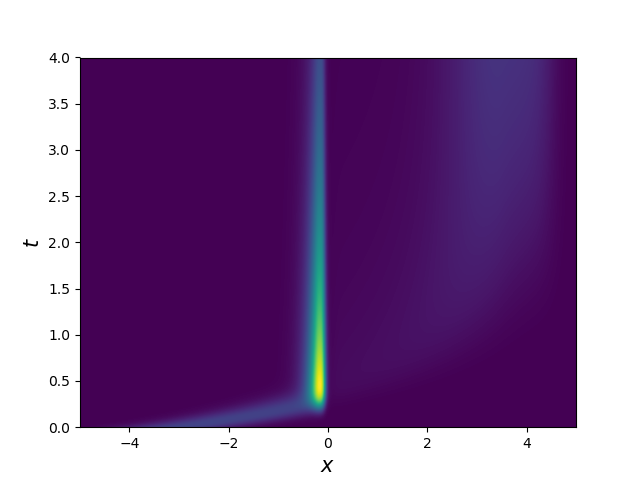}
    }
    \caption{The influence of finite size effects upon the evolution of the mobile species are displayed. At $x=0$, the maximum of the immobile species, we observe the formation of a bump of the mobile species. As the radii of red and blue particles are increased (left to right) the mobile species peaks increasingly higher as it encounters the inactive species due to the inhibited traversal caused by the size exclusion effects.}
    \label{fig:voleffects_inhibition}
\end{figure}

\begin{figure}[ht!]
    \centering
    \subfigure[$\delta_2= \delta_3 = 1 \times 10^{-1}$.]{
    \includegraphics[width=0.3\textwidth]{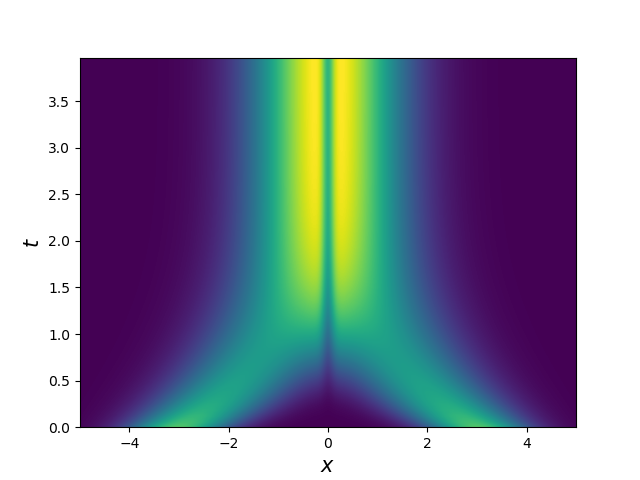}
    }
    \subfigure[$\delta_2= \delta_3 = 2 \times 10^{-1}$.]{
    \includegraphics[width=0.3\textwidth]{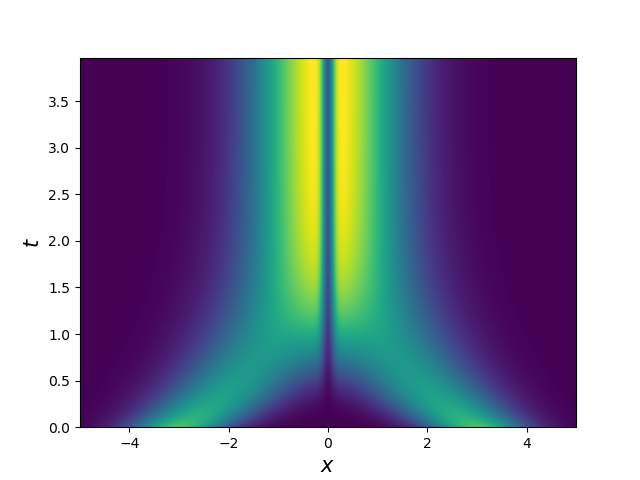}
    }
    \subfigure[$\delta_2= \delta_3 = 3 \times 10^{-1}$.]{
    \includegraphics[width=0.3\textwidth]{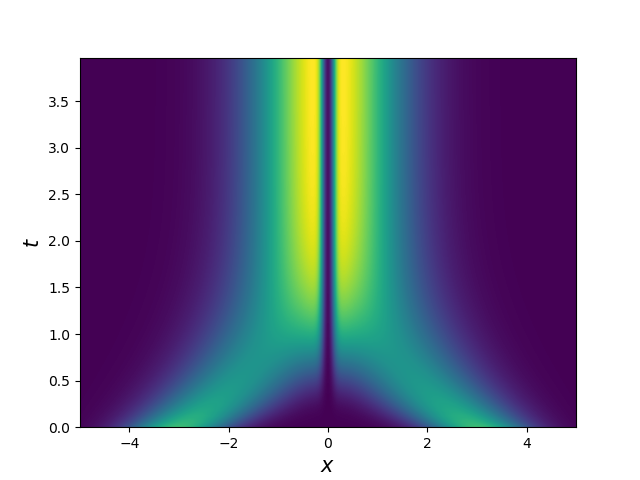}
    }
    \caption{Again, we can observe the impact of the size exclusion effect. The immobile species is located at $x=0$. As we increase the size of the particles of both species (left to right) the mobile species attains its maximum at a lower level. This is in perfect agreement with the physial derivation of the model as the size exclusion affects the density in regions of coexistence, i.e., around $x=0$.}
    \label{fig:voleffects_inhibition2}
\end{figure}

\begin{figure}
    \centering
    \includegraphics[width=0.6\textwidth]{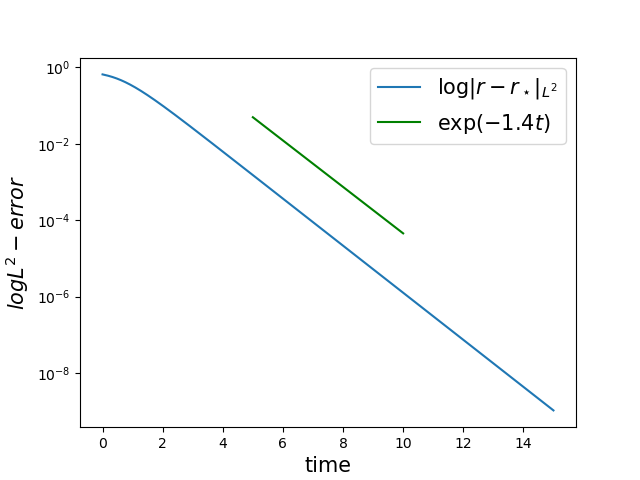}
    \caption{The $L^2$-distance (blue) between the evolution and the (numerically computed) minimiser $u^\star$ on a semi-log scale. A line of slope $-1.4$ (green) is superimposed to highlight the exponential convergence rate. The parameter choice corresponds to that of Figure \ref{fig:voleffects_inhibition2} (a).}
    \label{fig:expconvoneD}
\end{figure}

\subsection{Two Dimensional Explorations}

In this subsection we set $\Omega = (-1.75,1.75)^2$ and $\delta_2 = \delta_3 = 10^{-1}$. Note that this way both parameters are in the physical range in the sense of \cite{Bruna:2012wu}.

\subsubsection{Immobile species as porous medium}

In this example we assume the passive species is spread `heterogeneously' according to the distribution $b(x,y) = N^{-1}(1 + \cos(5 x) + \cos(5y))$, where $N$ is chosen such that $\int_{\Omega} b(x,y)\, dxdy = 1$. Initially the active species is spread around the origin, i.e., $r_0(x,y)=\frac{1}{\sqrt{2 \pi \sigma^2}}\exp\left(- \frac{r^2}{2 \sigma^2}\right)$, with $\sigma^2 = 0.5$. Moreover, we also choose $\delta_1 = 10^{-1}$ and we assume the absence of any external forces. Thus the dynamics are dictated only by the internal dispersal of the mobile species and its interaction with the immobile species. As the time evolution continues the volume exclusion effects imposed on the active species by the passive species become visible. In the final state, complementary regions are occupied by the passive species and the active species, respectively; cf. Figure \ref{fig:sim_4_evolution}.
\begin{figure}[ht!]
	\centering
	\includegraphics[width=0.4\textwidth]{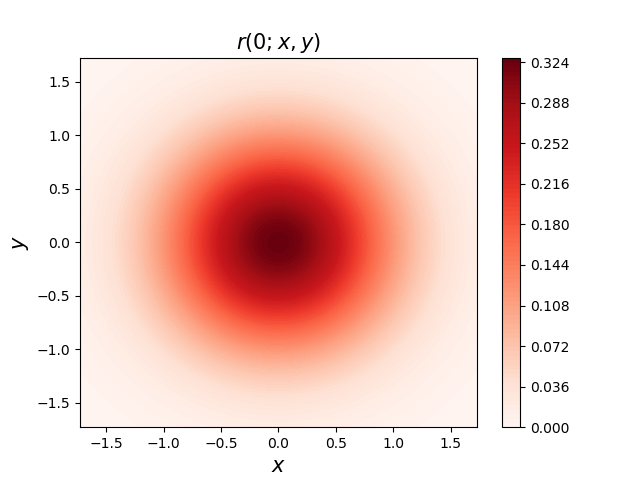}
	\includegraphics[width=0.4\textwidth]{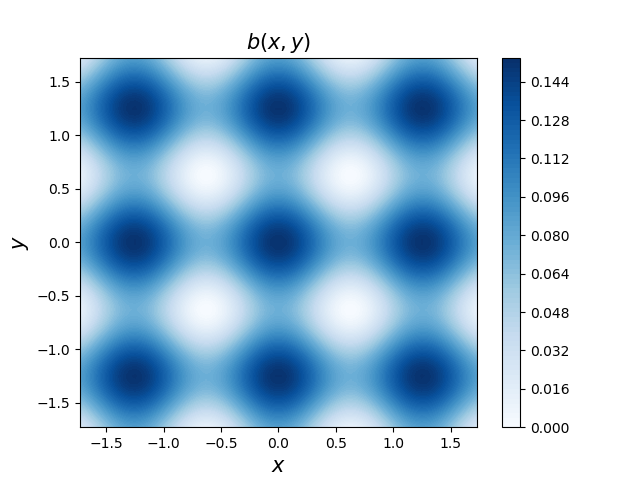}
	\includegraphics[width=0.4\textwidth]{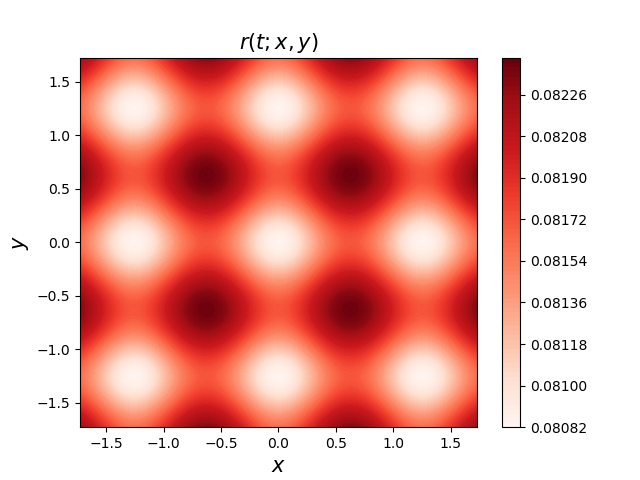}
	\includegraphics[width=0.4\textwidth]{sim_4_b_immobile}
	\caption{Initially the active species (red, left) is spread around the origin while the passive species is spread `heterogeneously' (blue, right). As the time evolution continues the volume exclusion effects imposed on the active species by the passive species become visible. In the final state, complementary regions are occupied by the passive species and the active species, respectively.}
	\label{fig:sim_4_evolution}
\end{figure}
After some time the evolution slows down as the mobile species reaches a stationary state. In Figure \ref{fig:sim_4_asymptotics} we show the numerical long-time asymptotics on a semi-log scale. It can be seen that the dynamic relaxes to the numerically computed stationary state at an exponential rate.
\begin{figure}
	\centering
	\includegraphics[width=0.6\textwidth]{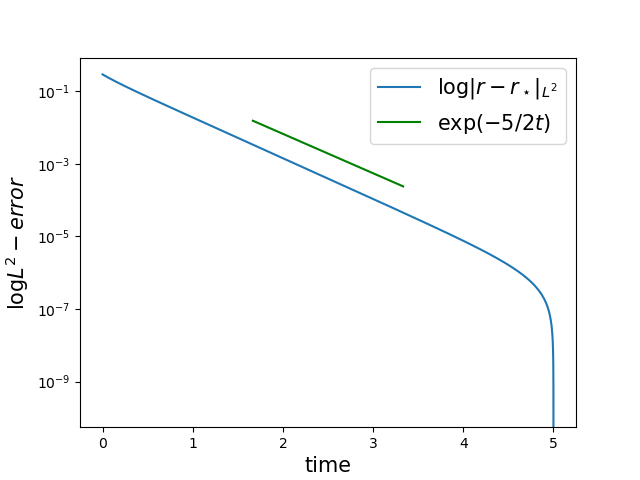}
	\caption{The $L^2$-distance (blue) between the evolution and the (numerically computed) minimiser $u^\star$ on a semi-log scale. A line of slope $-5/2$ (green) is superimposed to highlight the exponential convergence rate.}
	\label{fig:sim_4_asymptotics}
\end{figure}

\subsubsection{Immobile species as barrier}

In this section we present another stunning example of how the immobile species can inhibit the evolution of the mobile species. We set $\delta_1 = 10^{-2}$ and consider the initial datum
\begin{align*}
    r_0(x,y)=\frac{1}{\sqrt{2 \pi \sigma^2}}\exp\left(- \frac{(x+3/4)^2 + (y+3/4)^2}{2 \sigma^2}\right),
\end{align*}
with $\sigma = 0.3$. The immobile species is fixed and given by
\begin{align*}
    b(x,y) = &N^{-1}\left(\sin\left(5 \rho(x,y) \right)\right) \chi_{\{-7/4 +7\pi /5 \leq \rho(x,y) \leq -7/4 + 7.5\pi /5\}}\\[0.3em]
			&\times 
				(1 + \cos(4\pi/3(y+1)))\chi_{\{-1\leq y \leq 1\}}(x,y)\\[0.3em]
			&\times 
				(1 + \cos(4\pi/3(x+1)))\chi_{\{-1\leq x \leq 1\}}(x,y),
\end{align*}
where $\rho(x,y) = \sqrt{(x+7/4)^2 + (y+7/4)^2}$ and $N$ normalises the mass to one; cf. Figure \ref{fig:sim_8_b_acts_as_barrier}.
\begin{figure}
	\centering
	\includegraphics[width=0.45\textwidth]{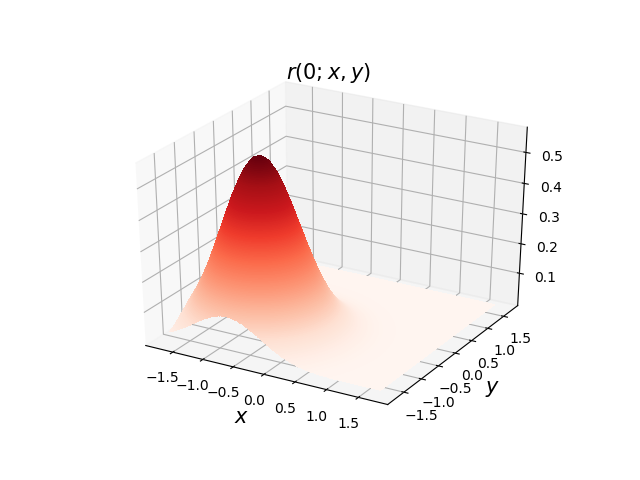}
	\includegraphics[width=0.45\textwidth]{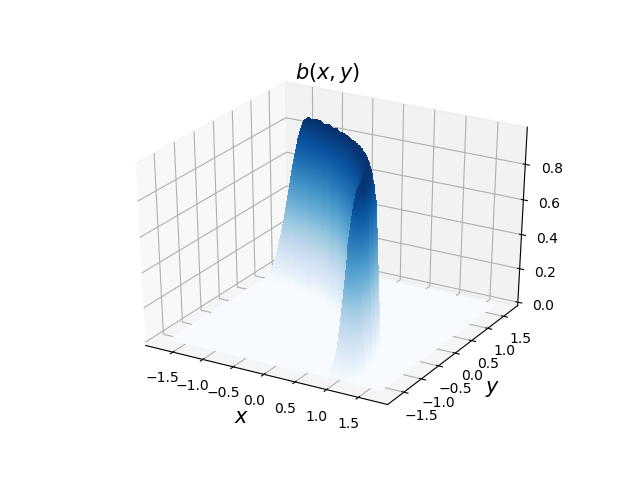}
	\caption{Initial data of the active species (red, left) and the immobile species (blue, right) acting as a barrier.}
	\label{fig:sim_8_b_acts_as_barrier}
\end{figure}
\newpage
Below we consider two different choices of external potentials: $V_\mathrm{s}(x) =  \frac1{20} \left((x-7/4)^2 + (y-7/4)^2\right)$ or $V_\mathrm{w}(x) = 1/10 \, V_\mathrm{s}(x)$. In the subsequent simulations the immobile species acts as a barrier as the active species tries to reach the minimum of the external potential centred at the upper right corner of the domain. 

In Figure \ref{fig:barrier_strong_potential} it can be observed that the motion of the active species is slowed down and that its density is lower in the region occupied by the passive species compared to the unoccupied regions. Since the potential is relatively strong the active species moves through the immobile one only slightly changing its shape. In Figure \ref{fig:barrier_weak_potential} we rescaled the potential by a factor of ten. 
Again, we see that the motion of the mobile species is slowed down and that its density is reduced in the region occupied by the passive species in comparison with unoccupied regions. Since, now,  the potential is relatively weak the motion of the active species is inhibited by the passive species which incentivises circumnavigating the immobile species instead of penetrating it. This behaviour is reflected in the numerical simulation; cf. Figure \ref{fig:barrier_weak_potential}. Compared to Figure \ref{fig:barrier_strong_potential} the active species closes in from the side as it approaches  the minimum of $V$, rather than directly moving into it.

\begin{figure}
	\centering
	\subfigure[$t=0$]{
		\includegraphics[width=0.3\textwidth]{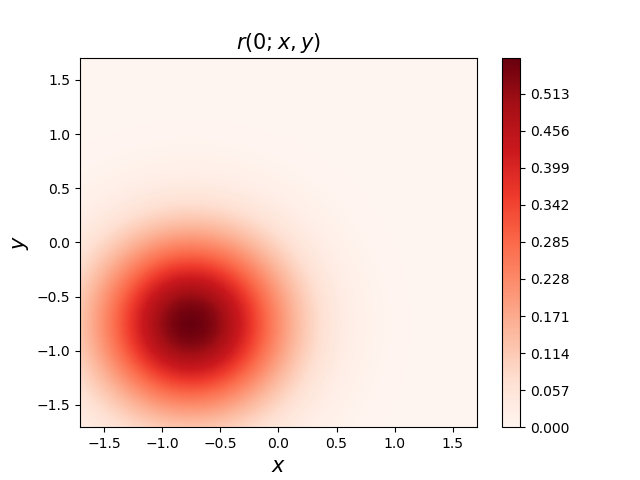}
	}
	\subfigure[$t=0.01$]{
		\includegraphics[width=0.3\textwidth]{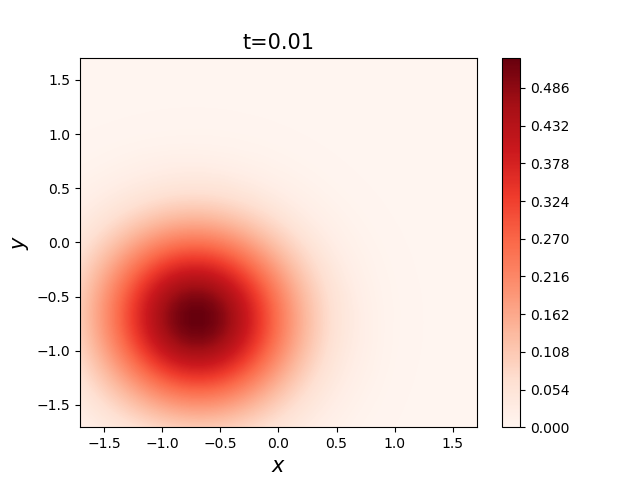}
	}
	\subfigure[$t=0.05$]{
		\includegraphics[width=0.3\textwidth]{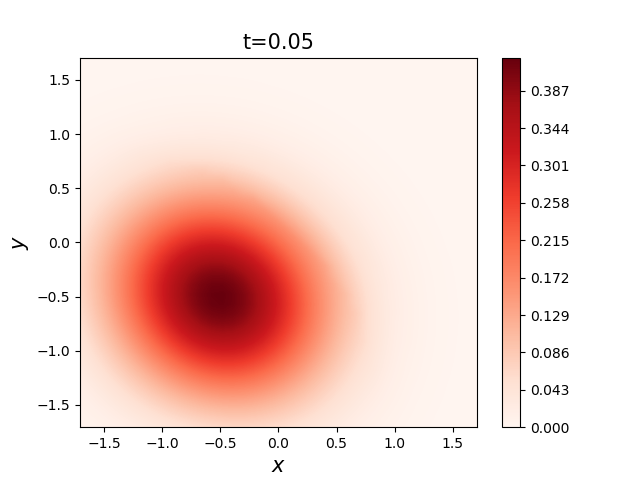}
	}
	\subfigure[$t=0.1$]{
		\includegraphics[width=0.3\textwidth]{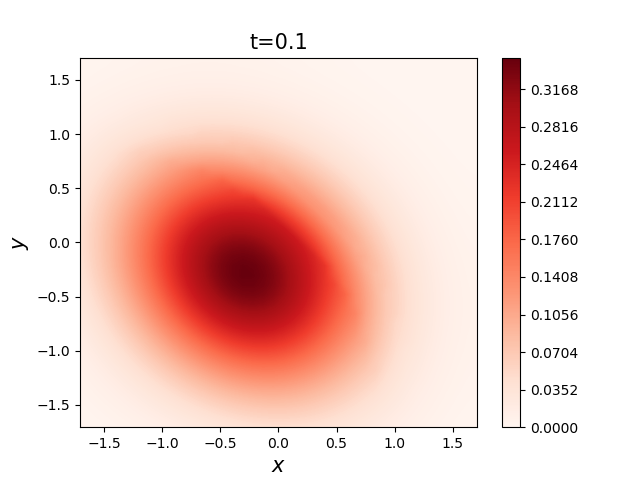}
	}
	\subfigure[$t=0.15$]{
		\includegraphics[width=0.3\textwidth]{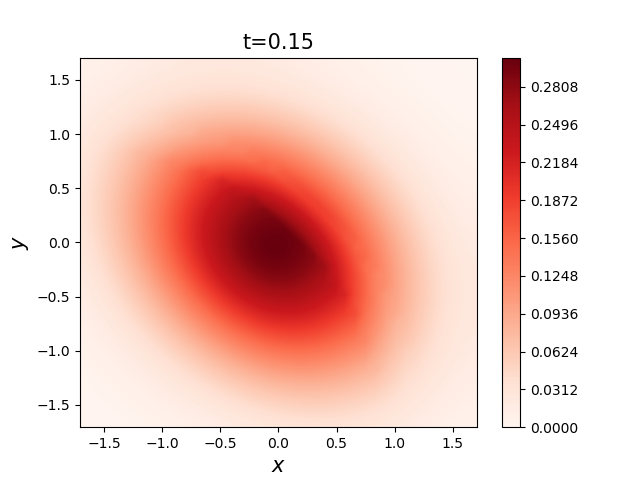}
	}
	\subfigure[$t=0.2$]{
		\includegraphics[width=0.3\textwidth]{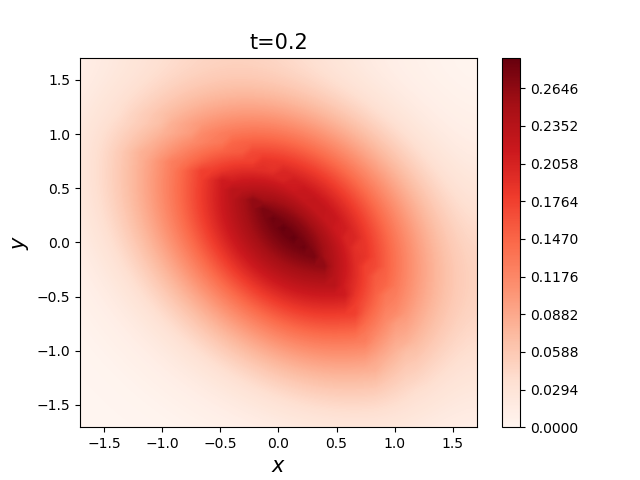}
	}
	\subfigure[$t=0.25$]{
		\includegraphics[width=0.3\textwidth]{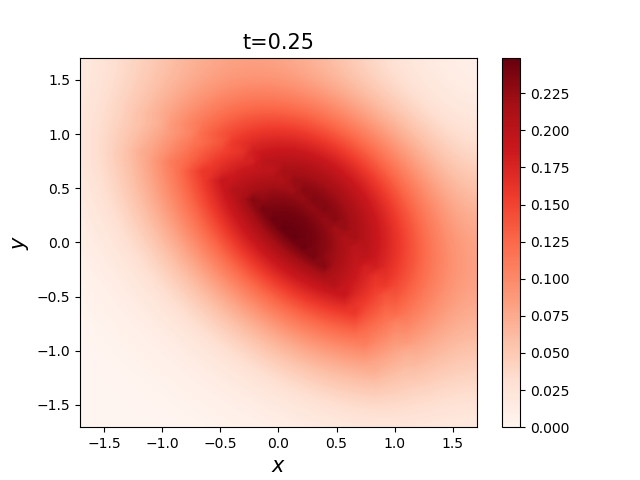}
	}
	\subfigure[$t=0.3$]{
		\includegraphics[width=0.3\textwidth]{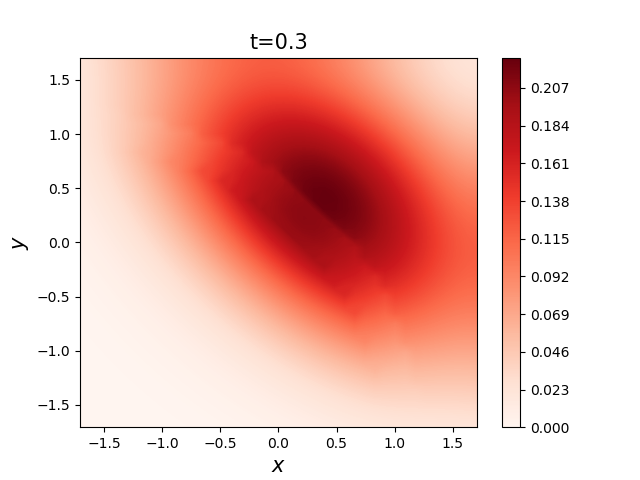}
	}
	\subfigure[$t=0.35$]{
		\includegraphics[width=0.3\textwidth]{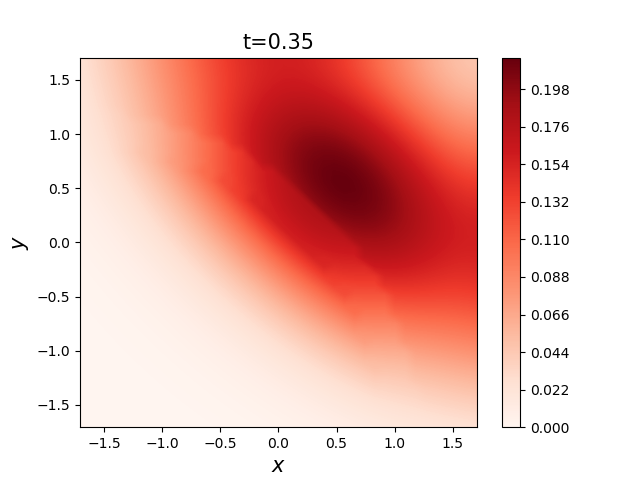}
	}
	\subfigure[$t=0.4$]{
		\includegraphics[width=0.3\textwidth]{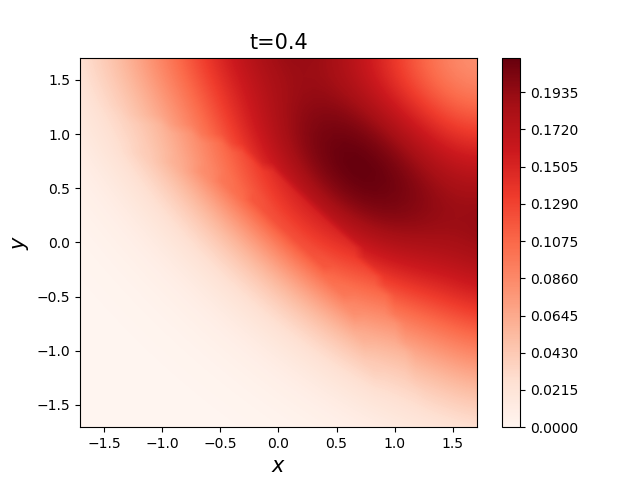}
	}
	\subfigure[$t=0.45$]{
		\includegraphics[width=0.3\textwidth]{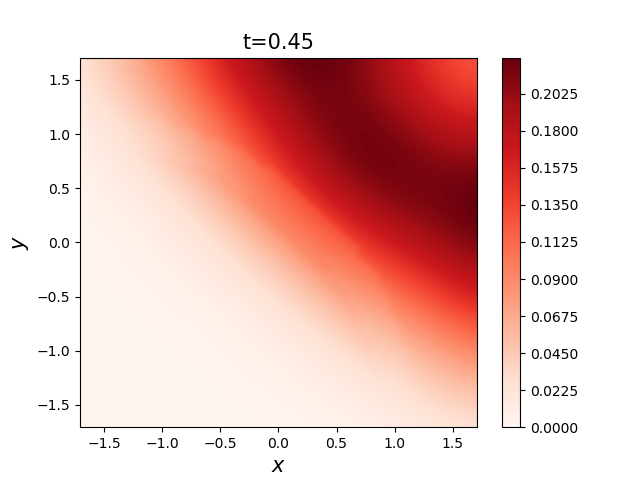}
	}
	\subfigure[$t=0.5$]{
		\includegraphics[width=0.3\textwidth]{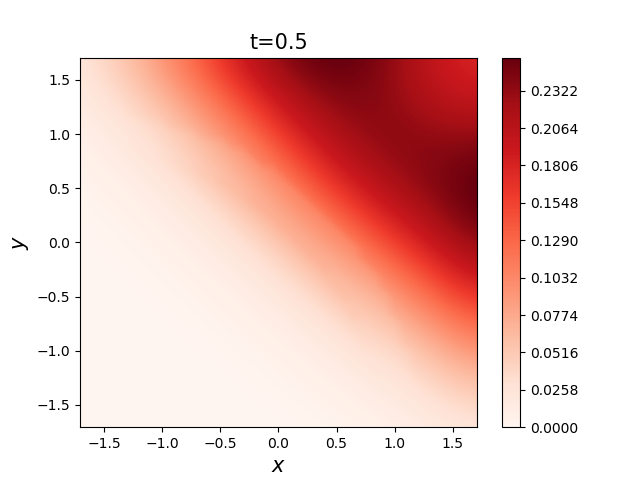}
	}
	\caption{Evolution of the mobile species, $r$, with $b$ acting as a barrier for the potential $V= V_\mathrm{s}$.}
	\label{fig:barrier_strong_potential}
\end{figure}

\begin{figure}
	\centering
	\subfigure[$t=0$]{
		\includegraphics[width=0.3\textwidth]{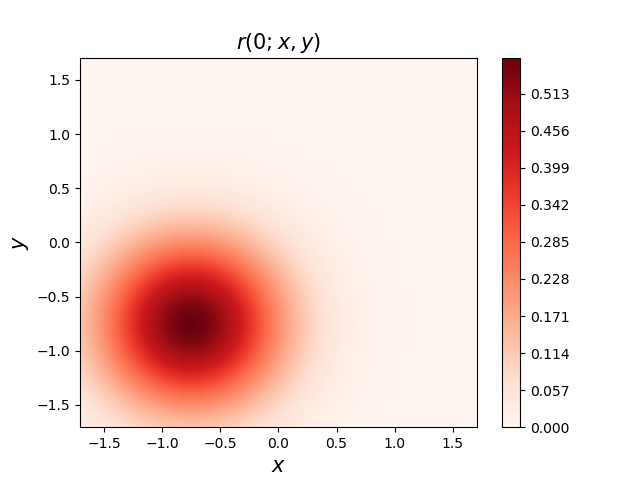}
	}
	\subfigure[$t=0.1$]{
		\includegraphics[width=0.3\textwidth]{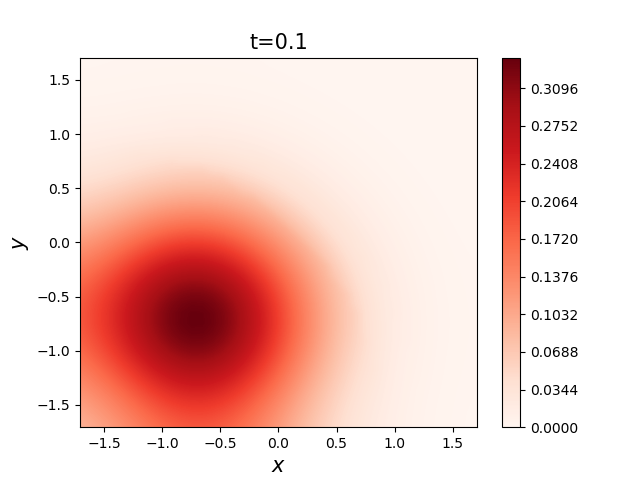}
	}
	\subfigure[$t=0.3$]{
		\includegraphics[width=0.3\textwidth]{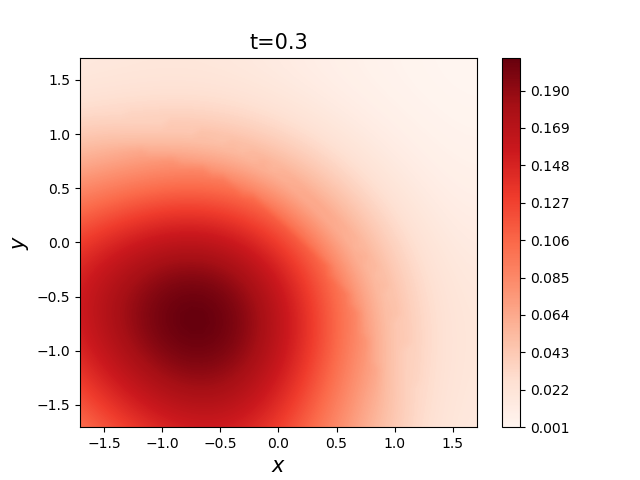}
	}
	\subfigure[$t=0.4$]{
		\includegraphics[width=0.3\textwidth]{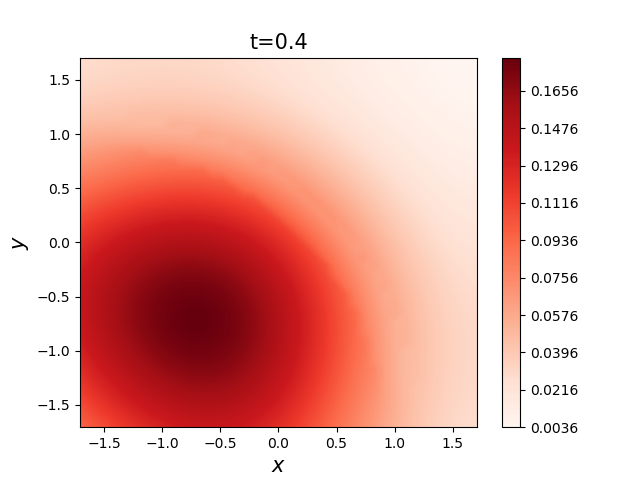}
	}
	\subfigure[$t=0.6$]{
		\includegraphics[width=0.3\textwidth]{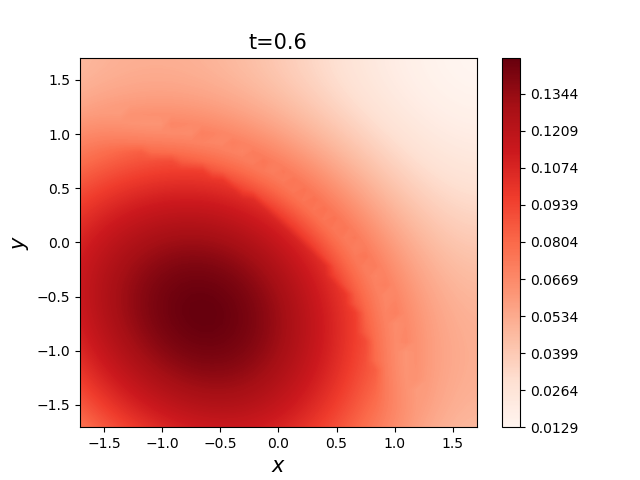}
	}
	\subfigure[$t=0.7$]{
		\includegraphics[width=0.3\textwidth]{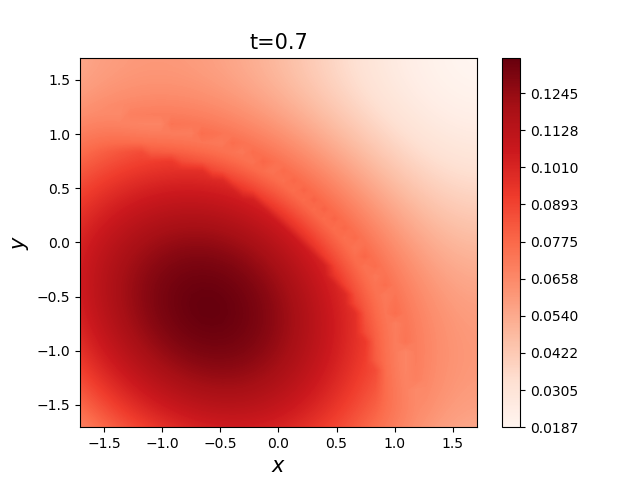}
	}
	\subfigure[$t=0.8$]{
		\includegraphics[width=0.3\textwidth]{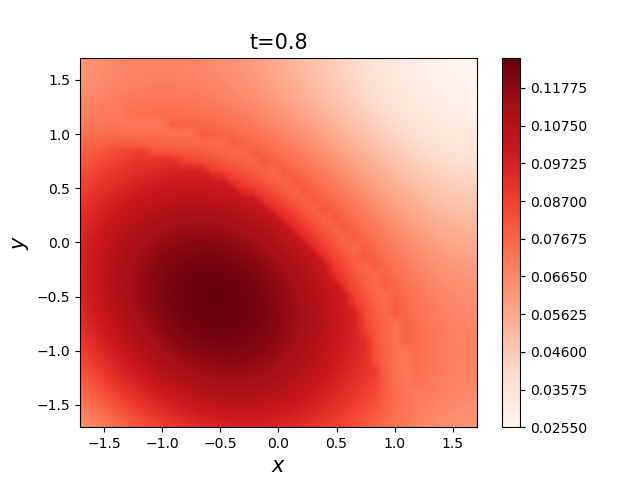}
	}
	\subfigure[$t=0.9$]{
		\includegraphics[width=0.3\textwidth]{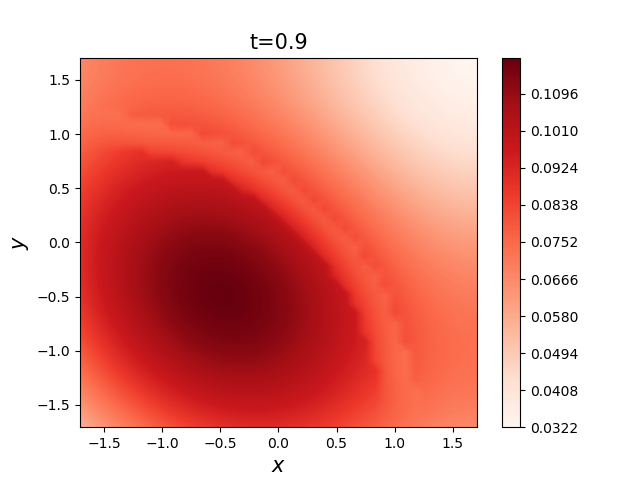}
	}
	\subfigure[$t=1.0$]{
		\includegraphics[width=0.3\textwidth]{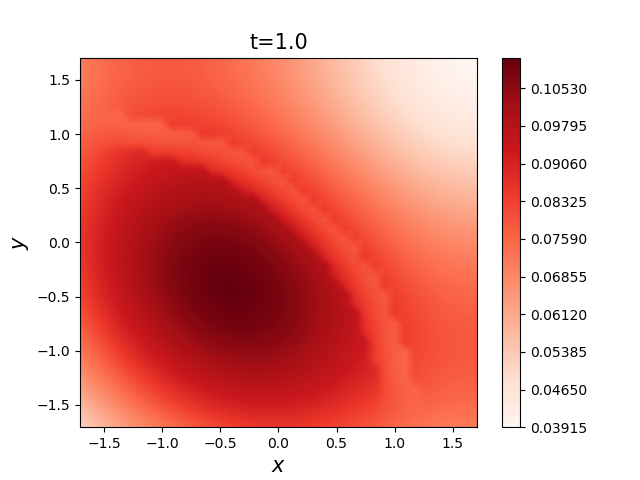}
	}
	\subfigure[$t=1.1$]{
		\includegraphics[width=0.3\textwidth]{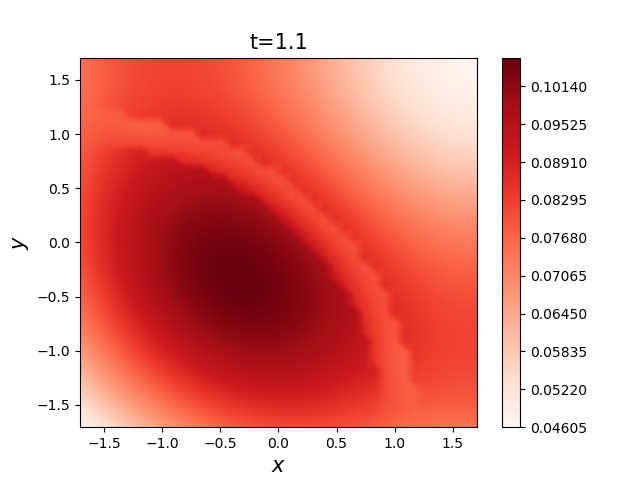}
	}
	\subfigure[$t=1.2$]{
		\includegraphics[width=0.3\textwidth]{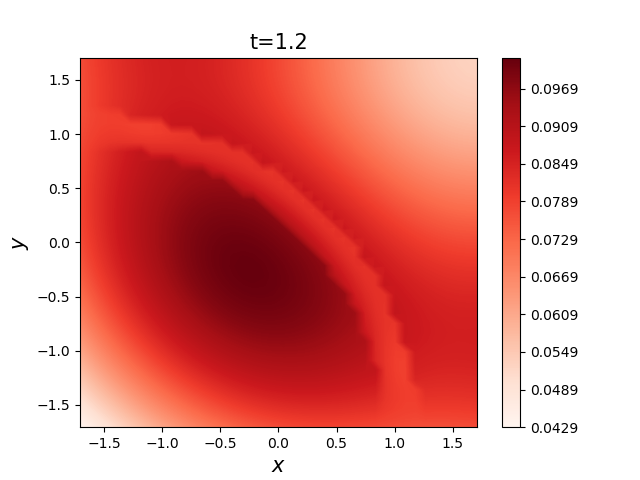}
	}
	\subfigure[$t=1.3$]{
		\includegraphics[width=0.3\textwidth]{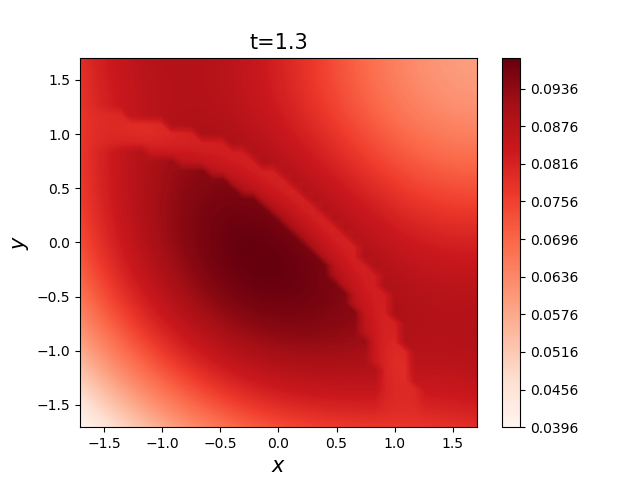}
	}
	\caption{The active species is incentivised to move around the passive species.}
	\label{fig:barrier_weak_potential}
\end{figure}

\FloatBarrier


\textbf{Acknowledgements.\,}
The basic ideas of this paper were developed at the Oxford working group meeting \quot{Asymptotic gradient flows} in 2017. The authors are very grateful to Maria Bruna, Jos\'e Antonio Carrillo, Jon Chapman and Ende S\"uli for the inspiring and valuable discussions during this meeting. M.T. Wolfram acknowledges partial support by the Austrian Academy of
Sciences via the New Frontier's Group NST-001 and by the EPSRC vis the First grant EP/P01240X/1. H. Ranetbauer was funded by the Austrian Science Fund (FWF) project F 65.


\bibliographystyle{amsplain}

\begin{thebibliography}{10}

\bibitem{acquistapace1988quasilinear}
P.~Acquistapace and B.~Terreni, \emph{On quasilinear parabolic systems},
  Mathematische Annalen \textbf{282} (1988), no.~2, 315--335.

\bibitem{adams2011large}
S.~Adams, N.~Dirr, M.~A. Peletier, and J.~Zimmer, \emph{From a large-deviations
  principle to the {W}asserstein gradient flow: a new micro-macro passage},
  Communications in Mathematical Physics \textbf{307} (2011), no.~3, 791.

\bibitem{alasio2018stability}
L.~Alasio, M.~Bruna, and Y.~Capdeboscq, \emph{Stability estimates for systems
  with small cross-diffusion}, ESAIM: Mathematical Modelling and Numerical
  Analysis \textbf{52} (2018), no.~3, 21109 -- 1135.

\bibitem{alasio2019global}
L.~Alasio and S.~Marchesani, \emph{Global existence for a class of viscous
  systems of conservation laws}, Nonlinear Differential Equations and
  Applications \textbf{26} (2019), no.~32.

\bibitem{amann1989dynamic}
H.~Amann, \emph{Dynamic theory of quasilinear parabolic systems}, Mathematische
  Zeitschrift \textbf{202} (1989), no.~2, 219--250.

\bibitem{arnold2001convex}
A.~Arnold, P.~Markowich, G.~Toscani, and A.~Unterreiter, \emph{On convex
  {S}obolev inequalities and the rate of convergence to equilibrium for
  {Fokker--Planck} type equations}, Communications in Partial Differential
  Equations \textbf{26} (2001), no.~1-2.

\bibitem{bakry1985}
D.~Bakry and M.~{\'E}mery, \emph{Diffusions hypercontractives}, S{\'e}minaire
  de Probabilit{\'e}s XIX 1983/84 (Berlin, Heidelberg) (Jacques Az{\'e}ma and
  Marc Yor, eds.), Springer Berlin Heidelberg, 1985, pp.~177--206.

\bibitem{Berendsen2019}
J.~Berendsen, M.~Burger, V.~Ehrlacher, and J.-F. Pietschmann, \emph{Uniqueness
  of strong solutions and weak--strong stability in a system of cross-diffusion
  equations}, Journal of Evolution Equations.

\bibitem{bessemoulin2012finite}
M.~Bessemoulin-Chatard and F.~Filbet, \emph{A finite volume scheme for
  nonlinear degenerate parabolic equations}, SIAM Journal on Scientific
  Computing \textbf{34} (2012), no.~5, B559--B583.

\bibitem{bodnar2005}
M.~Bodnar and J.~J.~L. Velazquez, \emph{Derivation of macroscopic equations for
  individual cell-based models: a formal approach}, Mathematical Methods in the
  Applied Sciences \textbf{28} (2005), no.~15, 1757--1779.

\bibitem{bruna2017asymptotic}
M.~Bruna, M.~Burger, H.~Ranetbauer, and M.-T. Wolfram, \emph{Asymptotic
  gradient flow structures of a nonlinear {F}okker-{P}lanck equation}, arXiv
  preprint arXiv:1708.07304 (2017).

\bibitem{bruna2017cross}
\bysame, \emph{Cross-diffusion systems with excluded-volume effects and
  asymptotic gradient flow structures}, Journal of Nonlinear Science
  \textbf{27} (2017), no.~2, 687--719.

\bibitem{Bruna:2012wu}
M.~Bruna and S.~J. Chapman, \emph{{Diffusion of multiple species with
  excluded-volume effects}}, The Journal of Chemical Physics \textbf{137}
  (2012), no.~20, 204116--204116--16.

\bibitem{Bruna:2012cg}
\bysame, \emph{{Excluded-volume effects in the diffusion of hard spheres}},
  Phys. Rev. E \textbf{85} (2012), no.~1.

\bibitem{Burger:2010gb}
M.~Burger, M.~Di~Francesco, J.-F. Pietschmann, and B.~Schlake, \emph{{Nonlinear
  {C}ross-{D}iffusion with {S}ize {E}xclusion}}, SIAM Journal on Mathematical
  Analysis \textbf{42} (2010), no.~6, 2842.

\bibitem{burger2016}
M.~Burger, S.~Hittmeir, H.~Ranetbauer, and M.-T. Wolfram, \emph{Lane formation
  by {S}ide-{S}tepping}, SIAM Journal on Mathematical Analysis \textbf{48}
  (2016), no.~2, 981--1005.

\bibitem{burger2012nonlinear}
M.~Burger, B.~Schlake, and M.-T. Wolfram, \emph{Nonlinear
  {P}oisson--{N}ernst--{P}lanck equations for ion flux through confined
  geometries}, Nonlinearity \textbf{25} (2012), no.~4, 961.

\bibitem{carrillo2018fvconvergence}
J.~A. Carrillo, F.~Filbet, and M.~Schmidtchen, \emph{Convergence of a finite
  volume scheme for a system of interacting species with cross-diffusion},
  arXiv preprint arXiv:1804.04385 (2018).

\bibitem{CHS18}
J.~A. Carrillo, Y.~Huang, and M.~Schmidtchen, \emph{Zoology of a {N}onlocal
  {C}ross-{D}iffusion {M}odel for {T}wo {S}pecies}, SIAM Journal on Applied
  Mathematics \textbf{78} (2018), no.~2.

\bibitem{carrillo2016diffeo}
J.~A Carrillo, H.~Ranetbauer, and M.-T. Wolfram, \emph{Numerical simulation of
  nonlinear continuity equations by evolving diffeomorphisms}, Journal of
  Computational Physics \textbf{327} (2016), 186--202.

\bibitem{desvillettes2015entropic}
L.~Desvillettes, T.~Lepoutre, A.~Moussa, and A.~Trescases, \emph{On the
  entropic structure of reaction-cross diffusion systems}, Communications in
  Partial Differential Equations \textbf{40} (2015), no.~9, 1705--1747.

\bibitem{difrancesco2018nonlinear}
M.~Di~Francesco, A.~Esposito, and S.~Fagioli, \emph{Nonlinear degenerate
  cross-diffusion systems with nonlocal interaction}, Nonlinear Analysis
  \textbf{169} (2018), 94--117.

\bibitem{dinezza2012hitchhiker}
E.~Di~Nezza, G.~Palatucci, and E.~Valdinoci, \emph{Hitchhiker's guide to the
  fractional {S}obolev spaces}, Bulletin des Sciences Math{\'e}matiques
  \textbf{136} (2012), no.~5, 521--573.

\bibitem{gavish2019large}
Nir Gavish, Pierre Nyquist, and Mark Peletier, \emph{Large deviations and
  gradient flows for the brownian one-dimensional hard-rod system}, 2019.

\bibitem{han2011elliptic}
Q.~Han and F.~Lin, \emph{Elliptic partial differential equations}, vol.~1,
  American Mathematical Soc., 2011.

\bibitem{jungel2016entropy}
A.~J{\"u}ngel, \emph{Entropy methods for diffusive partial differential
  equations}, Springer.

\bibitem{jungel2015boundedness}
\bysame, \emph{The boundedness-by-entropy method for cross-diffusion systems},
  Nonlinearity \textbf{28} (2015), no.~6, 1963.

\bibitem{ladyzhenskaia1988linear}
O.~A. Ladyzhenskaia, V.~A. Solonnikov, and N.~N. Ural'tseva, \emph{Linear and
  quasi-linear equations of parabolic type}, vol.~23, American Mathematical
  Soc., 1988.

\bibitem{matthes2014}
D.~Matthes and H.~Osberger, \emph{Convergence of a variational {L}agrangian
  scheme for a nonlinear drift diffusion equation}, ESAIM: Mathematical
  Modelling and Numerical Analysis \textbf{48} (2014), no.~3, 697–726.

\bibitem{payne1960optimal}
L.~E. Payne and H.~F. Weinberger, \emph{An optimal poincar{\'e} inequality for
  convex domains}, Archive for Rational Mechanics and Analysis \textbf{5}
  (1960), no.~1, 286--292.

\bibitem{perthame2015parabolic}
B.~Perthame, \emph{Parabolic equations in biology}, Parabolic Equations in
  Biology, Springer, 2015, pp.~1--21.

\bibitem{simpson2009multi}
M.~J. Simpson, K.~A. Landman, and B.~D. Hughes, \emph{Multi-species simple
  exclusion processes}, Physica A: Statistical Mechanics and its Applications
  \textbf{388} (2009), no.~4, 399--406.

\bibitem{troianiello}
G.~M. Troianiello, \emph{Elliptic differential equations and obstacle
  problems}, Springer Science \& Business Media, 2013.

\bibitem{zamponi2017analysis}
N.~Zamponi and Ansgar J., \emph{Analysis of degenerate cross-diffusion
  population models with volume filling},  \textbf{34} (2017), no.~1, 1--29.

\bibitem{ziemer2012weakly}
W.~P. Ziemer, \emph{Weakly differentiable functions: Sobolev spaces and
  functions of bounded variation}, vol. 120, Springer Science \& Business
  Media, 1989.

\end{thebibliography}

\end{document}